\documentclass[a4paper,12pt]{amsart}
\usepackage{hyperref}
\usepackage[utf8]{inputenc}
\usepackage{graphicx}
\usepackage{color}
\usepackage{booktabs}

\theoremstyle{plain}
\newtheorem{thm}{Theorem}
\newtheorem{cor}{Corollary}

\newtheorem{pro}{Proposition}
\newtheorem{lem}{Lemma}
\newtheorem{ass}{Assumption}
\theoremstyle{remark}
\newtheorem{rem}{Remark}
\newtheorem{ex}{Example}
\theoremstyle{definition}

\newcommand{\bbR}{\mathbb{R}}

\newcommand{\bbN}{\mathbb{N}}

\newcommand{\de}{\mathrm d}
\newcommand{\weak}\rightharpoonup

\newcommand{\Id}{\operatorname{Id}}
\newcommand{\inv}{^{-1}}

\newcommand{\qms}{Q_m^\sigma}
\newcommand{\qm}{Q_m}

\newcommand{\spsig}[2]{\left\langle #1,#2\right\rangle_\sigma}
\def\cprime{$'$}
\newcommand{\lr}[1]{\left(#1\right)}
\newcommand{\abs}[1]{\left\vert #1 \right\vert}
\newcommand{\norm}[2]{\left\Vert #1 \right\Vert _{#2}}
\newcommand{\set}[1]{\left\{#1\right\}}
\newcommand{\scalar}[2]{\langle #1,#2\rangle}
\newcommand{\yd}{y^\delta}
\newcommand{\xad}{x^\delta_\alpha}
\newcommand{\xnull}{x_0}
\newcommand{\res}{\operatorname{Res}}
\newcommand{\co}{c}

\title[Discretized Lavrent\cprime ev regularization]{Discretized
  Lavrent\cprime ev
  regularization for the autoconvolution equation} 
\author{Steven B\"urger}
\address{Technische Universität Chemnitz, Department of Mathematics, 
D-09107 Chemnitz, Germany}
\email{steven.buerger@mathematik.tu-chemnitz.de}
\author{Peter Math\'e}
\address{Weierstra{\ss} Institute for Applied Analysis and
  Stochastics, Mohrenstra\ss e 39, 10117 Berlin,  Germany}
\email{peter.mathe@wias-berlin.de}
\date{\today}
\begin{document}
\begin{abstract}
Lavrent\cprime ev regularization for the autoconvolution equation was
considered by J. Janno in {\it Lavrent\cprime ev regularization of ill-posed problems
              containing nonlinear near-to-monotone operators with
              application to autoconvolution equation}, Inverse
            Problems, 16(2):333--348, 2000. 
Here this study is extended by considering discretization of the
Lavrent\cprime ev scheme by splines. It is shown how to maintain the known
convergence rate by an appropriate choice of spline spaces and a
proper choice of the discretization level. For piece-wise constant
splines the discretized equation allows for an explicit solver, in
contrast to using higher order splines. This is used to design a fast
implementation by means of post-smoothing, which
provides results, which are indistinguishable from results obtained by
direct discretization using cubic splines. 
\end{abstract}

\maketitle
{\noindent\small Key words: autoconvolution, inverse problem, Lavrent\cprime ev regularization\\[1ex]
}
\section{Introduction, problem formulation}
\label{sec:intro}

We shall study the numerical solution of the following autoconvolution
operator, say~$F:L_2(0,1)\to L_2(0,1)$, given as
\begin{equation}
  \label{eq:base-auto}
  [F(x)](s):=\int_{0}^{s} x(s-t)x(t)\de t,\quad s \in [0,1], 
\end{equation}
for real-valued functions. The solution to such equations attracted attention both from
applications, cf.~\cite{Schleicher198333,0266-5611-32-3-035002}, and from the mathematical
side, see e.g.~\cite{MR1269013,MR1402101,MR3306115},
and~\cite{MR3353600}, where different approaches are taken.

Instead of the exact equation~$F(\xnull) =
y_{0}$ we are given noisy data~$\yd$ given as
\begin{equation}
\label{eq:base-noisy}
 \yd = F(\xnull) + \delta\xi
\end{equation}
for given noisy right hand side.
The assumption on the noise will be specified, below.

To this end we use Lavrent\cprime ev regularization, and hence we choose some
a priori guess, say $x_{\ast}$, and a parameter~$\alpha>0$ such that
we look for some $x = \xad$ with
\begin{equation}
  \label{eq:lavrentev}
  \alpha(x_{\ast} - x) + \yd = F(x).
\end{equation}
Extending the analysis from~\cite{MR1766765} we want to use finite
dimensional approximations to solve Eq.~(\ref{eq:base-noisy}).  
To this end we consider linear approximations~$Q = Q_{m}$ with certain
approximation properties, to be specified later. Having chosen $Q_{m}$
we replace Eq.~(\ref{eq:lavrentev}) by
\begin{equation}
  \label{eq:Qbase}
  \alpha Q_{m}(x_{\ast} - x) + Q_{m}\yd = Q_{m} F(Q_{m}x).
\end{equation}
\begin{rem}
  This is a two-sided discretization, since we use a finite amount of
  data~$Q_{m}\yd$, and we also aim at representing the solution in the
  range of $Q_{m}$.
\end{rem}

The outline for the material is as follows. In
Section~\ref{sec:error-analysis} we provide the error
analysis for~(\ref{eq:Qbase}), extending Janno's original ideas, where
no discretization was involved. These results are then extended in
Section~\ref{sec:extension}. We also derive an explicit implementation
of the scheme~(\ref{eq:Qbase}). Finally we provide some numerical study, where we
discuss the control of the involved parameters, in particular the
choice of~$\alpha$, and the discretization level~$m$, both as
functions of the noise level~$\delta$, which is assumed to be known.
Most of the proofs are given in the appendix.

\section{Error analysis of the discretized regularization problem} 
\label{sec:error-analysis}

The goal is to formulate the main results. To do so we first introduce
the basic assumptions, then we derive some properties of the
autoconvolution operator. Special attention is paid on the
approximation of the value~$\xnull(0)$, since this will prove
important. The main results and some consequences are then given
in~\S~\ref{sec:main-result}.

\subsection{Assumptions}
\label{sec:assumptions}

The error analysis will be based on several assumptions, especially on
the noise, the (unknown) solution~$x$ in~(\ref{eq:base-noisy}), and
the discretization schemes, to be considered.

We shall consider two different assumptions on the noise.
\begin{ass}
  [noise]\label{ass:noise}
Suppose that we are given a noise level~$\delta>0$.
  \begin{enumerate}
  \item\label{ass:1} The element~$\xi$ obeys
\begin{math}\label{noise1}
 \norm{\xi}{\infty} \le \delta.
\end{math}
\item\label{ass:1prime}The element~$\xi$ obeys
\begin{math}\label{noise2}
 \norm{\xi}{2} \le \delta.
\end{math}
  \end{enumerate}
\end{ass}
The main result, Theorem~\ref{thm:main} is concerned with the first
case, Assumption~\ref{ass:noise}\eqref{ass:1}, 
and we shall briefly discuss the more relaxed
assumption~\ref{ass:noise}\eqref{ass:1prime} in its Corollary~\ref{cor:noise1prime}.

We turn to the solution smoothness.
For $x\in L_\infty(0,1)$ we will denote the essential supremum norm of
$x$ as $\|x\|_\infty$.    
\begin{ass}
  [solution smoothness]\label{ass:onx}
The unknown solution element~$\xnull$ is positive, it belongs to
$W^{2}_{\infty}(0,1)$, and 
$\norm{\xnull}{C^{1}(0,1)}\leq K$, where we equip the
space~$C^{1}(0,1)$ with norm~$\norm{x}{C^{1}(0,1)}:=
\max\set{\norm{x}{\infty},\norm{x^{\prime}}{\infty}}$. In particular
we assume that~$\xnull(0) >0$.
\end{ass}
 \begin{rem}\label{rem:data-nonneg}
  Since we assumed that the unknown solution is positive and the autoconvolution of a positive function is obviously nonnegative, we know that the exact data $y_0 := F(x_0)$ are nonnegative.
  For this reason we will assume that our noisy data are nonnegative as well, otherwise we would define new data $\tilde y^\delta$ by
  \begin{equation*}
   \tilde y^\delta(s) = \begin{cases}y^\delta(s) & \text{ if } y^\delta(s) \ge 0\\ 0 & \text{ else}\end{cases}
  \end{equation*}
  without increasing the noise level.
 \end{rem}

Following the study~\cite{MR1766765}, having fixed some $\sigma\geq 0$ (to be
specified later)  we equip $L_2(0,1)$ with the
inner product
\[\langle x,y \rangle_\sigma = \int_0^1 e^{-2\sigma t} x(t) y(t) \de
t, \]
and the corresponding (weighted) norm
\[
\|x\|_\sigma^2=\int_{0}^{1} e^{-2\sigma t} x(t)^2 \de t.
\]
Although the norms~$\norm{\cdot}{}$ and $\norm{\cdot}{\sigma}$ are
equivalent, i.e.,\ for fixed $\sigma>0$ we have that
  \begin{equation}\label{norm:equiv}
  e^{-\sigma}\|x\|_0 \le \|x\|_\sigma \le \|x\|_0,\quad x\in L_{2}(0,1), 
 \end{equation}
we shall occasionally denote the weighted Hilbert space by
$L_{2}^{\sigma}(0,1)$.
There is a natural isometry~$D_{\sigma}\colon L_{2}(0,1)\to
L_{2}^{\sigma}(0,1)$, given through the function~$f_{\sigma}(t) =
e^{\sigma t},\ t\in[0,1]$
as
\begin{equation}
  \label{eq:dsigma}
D_{\sigma}x = f_{\sigma }x,\quad x\in L_{2}(0,1). 
\end{equation}

Similarly, if $x\in L_{2}(0,1)$ is such that it is absolutely
continuous, and the weak derivatives are in~$L_{2}(0,1)$, then we
shall denote the weighted Sobolev Hilbert space of such elements by
$H^{\sigma}_{1}(0,1)$, equipped with norm
$$
\norm{x}{1,\sigma}:=\frac 1 2
\lr{\norm{x}{\sigma}^{2}+ \norm{x^{\prime}}{\sigma}^{2}}^{1/2}.
$$
Furthermore we introduce the operator norm $\norm . {\sigma \to \sigma}$ as
\begin{align*}
 \norm A {\sigma \to \sigma} := \norm{A\colon L_{2}^{\sigma}(0,1)\to
  L_{2}^{\sigma}(0,1)}{} = \sup_{x\in L_2^\sigma(0,1)}
  \frac{\norm{Ax}\sigma}{\norm x\sigma}. 
\end{align*}
for linear Operators $A:L_2^\sigma(0,1) \to L_2^\sigma(0,1)$.
\begin{rem}
Note that $\|.\|_\sigma$ reduces to  the standard $L_2$-norm for
$\sigma=~0$. 
Also, due to the above equivalence, the spaces~$H^{\sigma}_{1}(0,1)$
are equivalent to the usual Sobolev Hilbert spaces~$H^{1}(0,1)$.
Finally, we mention that functions~$x\in C^{1}(0,1)$ belong
to~$H^{\sigma}_{1}(0,1)$ and $\norm{x}{H_{1}^{\sigma}(0,1)}\leq
\norm{x}{C^{1}(0,1)}$.
\end{rem}

We turn to describing the approximation scheme, captured by
the approximation operators~$Q_{m},\ m=1,2,\dots$
\begin{ass}
  [approximation power]\label{ass:onQ}
We assume that we are given a sequence~$X_{m}\subset L_{2}(0,1),\ m\in\bbN$, of
finite dimensional subspaces, $\dim(X_{m})=m$ with orthogonal (in
$L_{2}^{\sigma}(0,1)$) projections~$\qms$ onto the spaces~$X_{m}$.  
There is a constant~$L<\infty$ such that for all $x\in H^{\sigma}_{1}(0,1)$
we have
\begin{equation}
\label{it:best-appr}
\inf\set{\norm{x - z}{\sigma},\quad z\in X_{m}} \leq \frac L
m\norm{x}{H^{\sigma}_{1}(0,1)},\quad  m\in\bbN.
\end{equation}
\end{ass}

We highlight the relations and approximation properties between the
unweighted ($\sigma=0$) and weighted ($\sigma>0$) spaces.
Specifically, any family of projection operators
$(Q_m)_{m\in\bbN}$, which fulfills assumption \ref{ass:onQ} for $\sigma=0$
can be used to construct projection operators $\qms$ for $\sigma>0$,
which also satisfy assumption \ref{ass:onQ}, with simple proof to be
found in the appendix. 
\begin{lem}\label{lem:02sigma}
  Let $X_{m}\subset L_{2}(0,1),\ m\in\bbN$ a sequence of
 finite dimensional subspaces with orthogonal (in
 $L_{2}(0,1)$) projections~$\qm$ onto the spaces~$X_{m}$,
 satisfying Assumption~\ref{ass:onQ} with a constant $L<\infty$.
 Moreover let $f_\sigma$ and $D_\sigma$ as in~(\ref{eq:dsigma}).
 Then for $\sigma>0$ the projection operators $\qms$ defined by
 \begin{equation}
  \qms := D_\sigma \qm D_\sigma^{\inv}
 \end{equation}
 are orthogonal projection operators w.r.t. $\langle.,.\rangle_\sigma$, for which
 Assumption~\ref{ass:onQ} holds with constant
 $\sqrt 2 \lr{1 + \sigma}L$, onto the spaces~$D_{\sigma}X_{m}$.
\end{lem}

We provide the following illustrative examples.
\begin{ex}[piece-wise constant splines]\label{ex:fsigma-constant}
We let~$\tilde X_{m}$ be the spline spaces~$\mathcal S^{m}_{1}$ of
piece-wise constant functions with 
respect to the equi-distant partition~$\Delta_{i}= [(i-1)/m,i/m),\ i=1,\dots,m$.
 The spaces~$X_{m}$ are then given as $X_{m}= D_{\sigma} (\tilde
 X_{m})$.

The approximation power for elements
in~$H^{1}(0,1)$  by piece-wise constant functions is known (cf.~\cite[Thm.~6.1, eq. (6.7)]{MR2348176}
with p=q=2) as
\begin{equation}
 \inf\set{\norm{x - z},\quad z\in X_{m}} \leq \frac L
 {m}\norm{x}{H_1(0,1)},\quad  m\in\bbN,  
\end{equation}
  
and hence the spaces~$\tilde X_{m}$
obey Assumption~\ref{ass:onQ} for $\sigma=0$ with constant~$L=1$. By virtue of
Lemma~\ref{lem:02sigma} this extends to the spaces~$X_{m}$. 
We mention that this approximation maps nonnegative functions to
nonnegative piece-wise constant functions.
\end{ex}

\begin{ex}
  [cubic splines]
As above we consider the equidistant partition, and we let~$\mathcal
S_{4}^{m}$ be the space of cubic splines, with corresponding
orthogonal (in~$L_{2}(0,1)$) projection~$Q_{m}$. 
Then Assumption~\ref{ass:onQ} holds, and  we refer to the
comprehensive monograph~\cite[Cor.~6.26]{MR2348176}
(with~$m=4,\sigma=1,p=q=2,r=0$) for~$\sigma=0$.
Therefore, the validity of Assumption~\ref{ass:onQ}
extends to $\sigma>0$. 

Actually, for functions which fulfill
Assumption~\ref{ass:onx} we even have the stronger assertion
$$
\inf\set{\norm{x - z}{\infty},\quad z\in X_{m}} \leq \frac L
{m^{2}}\norm{x}{W^{2}_{\infty}(0,1)},\quad  m\in\bbN, 
$$
see e.g.~\cite[Cor.~6.26]{MR2348176} (with~$m=4,\sigma=2,p=q=\infty,r=0$). However, we cannot
benefit from this additional approximation power in the overall
performance of the proposed Lavrent\cprime ev regularization.
\end{ex}

\subsection{Properties of the autoconvolution operator}
\label{sec:props}

In the subsequent analysis we will relate the nonlinear
autoconvolution equation~(\ref{eq:base-auto}) to the following linear
Volterra equation. 
We note that for $F$ from~(\ref{eq:base-auto}) its
Fr\`echet-derivative~$F^{\prime}(x)\colon L_{2}(0,1) \to L_{2}(0,1)$, at element~$x$, 
is given by
\begin{equation}
  \label{eq:frechet}
  [F^{\prime}(x)h](s) = 2\int_0^s x(s-t) h(t) \de t,\quad s \in [0,1],\ h\in L_{2}(0,1).
\end{equation}

Under the assumptions made on the solution~$\xnull$ the equation
\begin{equation}
  \label{eq:volterra}
F'(\xnull)\omega = \xnull - \xnull(0),
\end{equation}
has a solution, and we refer to~\cite[Lemma~4]{MR1766765}.

We bound the norm of the Fr\`echet derivative in the weighted Hilbert
spaces~$L_{2}^{\sigma}(0,1)$. 
\begin{lem}\label{lem:frechet} 
For the Fr\`echet-derivative~$F^{\prime}(u)$ of $F$ at $u\in L_2^{\sigma}(0,1)$
the estimate \[
\|F'(u)\colon L_{2}^{\sigma}(0,1)\to  L_{2}^{\sigma}(0,1)\|\le 2\|u\|_{\sigma},\quad u\in L_{2}^{\sigma}(0,1),
\] 
holds. Consequently, we also have that
\begin{equation}\label{F}
 \|F(x)\|_{\sigma}\le \|x\|_{\sigma}^2,\quad x\in L_{2}^{\sigma}(0,1).
 \end{equation}
\end{lem}
\begin{proof}
Let $u\in L_2^{\sigma}(0,1)$. Then we have
\begin{align*}
 \|F'(u)v\|_{\sigma}^2 =& \int_0^1 e^{-2\sigma s} \left(2\int_0^s u(s-t) v(t) \de t\right)^2 \de s\\
 =& 4\int_0^1 \left(\int_0^s e^{-\sigma(s-t)} u(s-t) e^{-\sigma t} v(t) \de t\right)^2 \de s\\
 \le& 4\int_0^1 \left(\int_0^s e^{-2\sigma(s-t)} u(s-t)^2\de t\right) \left(\int_0^s e^{-2\sigma t} v(t)^2 \de t\right) \de s\\
 \le& 4\int_0^1 \left(\int_0^1 e^{-2\sigma t} u(t)^2\de t\right) \left(\int_0^1 e^{-2\sigma t} v(t)^2 \de t\right) \de s\\
 =& 4\|u\|_{\sigma}^2 \|v\|_{\sigma}^2.
\end{align*}
The final assertions follows from the observation that $F(x) =
\frac12F^{\prime}(x) x$, which completes the proof.
\end{proof}

The following technical lemma will be used to prove the main result.
\begin{lem}[{cf.~\cite[{Lem.~3 \& 4}]{MR1766765}}]
\label{lem:suumary-sigma}
  Suppose that $\xnull$ obeys Assumption~\ref{ass:onx}. For each~$0 <
  \co<1$ there is
  some $\sigma_{0}\geq 0$ such that for $\sigma \geq \sigma_{0}$ we have that
  \begin{enumerate}
  \item $\scalar{F^{\prime}(\xnull) v}{v}_{\sigma}\geq 0,\quad v\in
    L_{2}(0,1)$, and
\item the solution~$\omega$ to~(\ref{eq:volterra}) obeys~$\norm{\omega}{\sigma}
  < \co$. 
  \end{enumerate}
\end{lem}

\subsection{Approximation of the initial value}
\label{sec:approx-init}

Recall, for the solution to
the equation~(\ref{eq:volterra}) to exist, 
we required to know the value~$\xnull(0)$. Since this is not the case,
we need to find a good approximation to it, based on the given
data~$\yd$. This is done  by averaging
with the approximating mapping~$Q_{m}^0$ from example
\ref{ex:fsigma-constant}, formulated  in the following proposition, again
with proof postponed to the Appendix.
\begin{pro}\label{pro:x0strich-appr}
 Let $\sigma>0$, $\delta\le 1$ and $m\in \bbN$ be fixed, with $m\geq
 1/\delta$.  If the noise obeys
 Assumption~\ref{ass:noise}\eqref{ass:1} then
 \begin{equation*}
 \left|\sqrt{\frac1{\sqrt \delta}{[Q_m\yd](\sqrt
       \delta)}}-\xnull(0)\right| 
\le \frac{\left(1 +
    {4\norm{\xnull}{C^{1}(0,1)}^{2}}\right)}{\xnull(0)}\sqrt\delta,
 \end{equation*}
 if $[Q_m\yd](\sqrt\delta)\ge 0$ and
 \begin{equation*}
 \abs{x_0(0)} \le \frac{\left(1 +
    {4\norm{\xnull}{C^{1}(0,1)}^{2}}\right)}{\xnull(0)}\sqrt\delta,
 \end{equation*}
 otherwise.
\end{pro}
Thus we use the approximation as found in Proposition~\ref{pro:x0strich-appr} to define the reference element
$x_{\ast}$. Specifically, given~$\delta>0$ we let $x_{\ast}$ be the
constant function defined as 
\begin{equation}
  \label{eq:xast}
x_{\ast}\equiv 
\begin{cases}
\sqrt{\frac1{\sqrt \delta}{[Q_m\yd](\sqrt \delta})},&
  \text { if }[Q_m\yd](\sqrt\delta)\ge 0\\
 0, & \text{ else.}
\end{cases}  
\end{equation}

The above approximation cannot be used for general noise which just
belongs to~$L_{2}(0,1)$.  We give the following result for this case.
\begin{pro}\label{pro:different-noise}
  Suppose that the noise obeys
  Assumption~\ref{ass:1}\eqref{ass:1prime}.  Then there is a
  constant~$D< \infty$ such that for~$ h:= \left(\frac23
    \|x_0\|^2_{C^1(0,1)}\right)^{-\frac25}\delta^{\frac25}$ we have that
\[
\left|\sqrt{\frac2{h^2}\int_0^h \yd(t)\de t}-\xnull(0)\right| \le
\frac{D }{\xnull(0)}\delta^{\frac25},\quad \text{as}\ \delta \to 0. 
\]
\end{pro}

\begin{rem}
  Estimation of the derivative of a function under $L_{2}$-noise has
  not been studied as often. The best results in this direction are
  presented in~\cite{MR2240638,MR3043622}. This will not immediately give results
  under the smoothness Assumption~\ref{ass:onx}. It is not clear to
  the authors, whether a reconstruction rate~$\delta^{1/2}$ is
  possible under~$L_{2}$-noise.
\end{rem}

\subsection{Main result}
\label{sec:main-result}

We now recall the equation~(\ref{eq:Qbase}), as
\begin{equation}\label{eq}
 Q(\alpha (x_*-x) + \yd - F(Q x))=0,
\end{equation}
where $Q$ is any, not necessarily orthogonal, projection. We are interested in its
solution. Therefore we introduce the family of linear 
operators~$H_\alpha:= \lr{\alpha \Id + Q F^{\prime}(\xnull)Q},\ \alpha>0$.
For projections~$Q$ onto some space~$X_{m}\subset X$, orthogonal in a
suited~$L_{2}^{\sigma}(0,1)$ the mapping~$H_{\alpha}$
is continuously invertible, and maps~$X_{m}$ into $X_{m}$, which is
easy to check. For given~$\alpha>0$, to be specified later, we
apply~$H_{\alpha}^{-1}$ to both sides. We see that 
\begin{align*}
 0 =& H_\alpha\inv Q\big(\alpha ( x_*-x) + y^\delta - F(Q x)\big)\\
 =& H_\alpha\inv Q \big( y^\delta - F(\xnull) -F'(\xnull)(Q x -
    \xnull) - F(Q x - \xnull) + \alpha ( x_*-x)\big)\\ 
 =& H_\alpha\inv \Big( -\big(Q F'(\xnull)(Q( x - \xnull) )+\alpha(Q
    x-\xnull)\big) + Q F'(\xnull)(\xnull-Q \xnull) \\ 
 &\qquad\ \ + Q \big( y^\delta - y_0 - F(Q x-\xnull )\big) + \alpha (Q  x_*-\xnull)\Big)\\
 =& \xnull -Q x + \\ 
  \begin{split}
H_\alpha\inv \Big( Q \big( F'(\xnull)(\xnull-Q \xnull) + y^\delta
  - y_0- 
F(Q_mx-\xnull )\big) + \alpha (Q  x_*-\xnull) \Big).     
  \end{split}
\end{align*}
This can be written as a fixed-point equation for the (continuous nonlinear)
mapping~$G\colon X\to X$, given by
\begin{equation}
  \label{eq:G-def}
  \begin{split}
  G(Q x) := \xnull + H_\alpha\inv \Big( Q \big( F'(\xnull)(\xnull-Q
  \xnull) + \yd - y_0\\
- F(Q_mx-\xnull )\big) + \alpha (Q  x_*-\xnull) \Big),\ x\in L_{2}^{\sigma}(0,1). 
  \end{split}
\end{equation}
The fixed-point equation we consider is now given as
\begin{equation}\label{fix}
 G(Q x ) = Q x,
\end{equation}
on some domain to be determined later.
The following holds true.
\begin{pro}\label{pro:contraction}
Let~$\sigma>0$ be such that~$F^{\prime}(\xnull)$ is accretive on~$L_{2}^{\sigma}(0,1)$.
Suppose that~$m\geq \sigma/\delta$, $\delta\leq 1$ and that the
element~$\omega$ as the solution to~(\ref{eq:volterra})
satisfies~$\norm{\omega}{\sigma} < 1/4$. If the noise obeys
Assumption~\ref{ass:noise}(\ref{ass:1}) then there is some $r>0$ such that the
  mapping~$G$ obeys
$$
G\lr{B(\xnull,r) \cap X_{m}} \subset B(\xnull,r) \cap X_{m}.
$$
Consequently it has a fixed point.
\end{pro}
We are now in the position to formulate the main result, 
\begin{thm}\label{thm:main}
Suppose that the noise obeys Assumption~\ref{ass:noise}\eqref{ass:1},
and that the true solution obeys Assumption~\ref{ass:onx}, and
that the discretization is with spaces~$X_{m}$ which fulfill
Assumption~\ref{ass:onQ}. Moreover let the assumptions of
Proposition \ref{pro:contraction} hold. There is a constant~$c>0$ such
that for~$\alpha:= c \sqrt\delta$,  Lavrent\cprime ev regularization with
discretization according to~(\ref{eq:Qbase}) has a solution~$\xad\in
X_{m}$ which obeys
$$
\norm{\xnull - \xad}{} \leq c \sqrt\delta,\quad \text{as}\ \delta\to 0.
$$
\end{thm}
Again, we postpone the proofs, both for
Proposition~\ref{pro:contraction} and Theorem~\ref{thm:main} to the Appendix.

We shall next mention the modification of the main results when the
noise obeys Assumption~\ref{ass:1}\eqref{ass:1prime}.
As the interested reader may check, the arguments used in the proofs
of Proposition~\ref{pro:contraction} and Theorem~\ref{thm:main} remain
valid, except that the optimal parameter choice strategy is different.
\begin{cor}
  \label{cor:noise1prime}
Under the assumptions of Theorem~\ref{thm:main}, but with noise
fulfilling Assumption~\ref{ass:1}\eqref{ass:1prime} we have the
following. There is a constant~$c>0$ such
that for  the parameter~$\alpha$ is chosen from
\begin{equation*}
 \alpha = c \delta^{\frac25}
\end{equation*}
this yields
\begin{equation*}
 \norm{\xnull - \xad}{} \leq c\delta^{\frac25},\quad \text{as}\
 \delta\to 0.
\end{equation*}
\end{cor}
\section{Extension}
\label{sec:extension}

We will extend the main results to the situation when the projections
used in~(\ref{eq}) are not orthogonal, but the corresponding
spaces~$X_{m}$ still obey Assumption~\ref{ass:onQ}.

As we know from Lemma~\ref{lem:suumary-sigma} we can choose
the value of $\sigma>0$ such that accretivity of the
operator~$F^{\prime}(\xnull)$ holds in~$L_{2}^{\sigma}(0,1)$. This, of
course, extends to accretivity of the mapping~$QF^{\prime}(\xnull)
Q$, whenever~$Q$ is an orthogonal projection in~$L_{2}^{\sigma}(0,1)$.

In some cases, the accretivity extends at the expense of an additional
correction term to nonorthogonal projections.
 The prototypical example is the projection onto
the piece-wise constant functions, which is orthogonal in
$L_{2}(0,1)$, but not in $L_{2}^{\sigma}(0,1)$. The
following 'closeness' can be established.

\begin{lem}
  \label{lem:qm-qms}
Consider the spaces~$\tilde X_{m}$ as in
Example~\ref{ex:fsigma-constant}, and let $Q_{m}:= Q_{m}^{0}$ be the
projection, which is orthogonal in $L_{2}(0,1)$. Then, for $m\geq
\sigma$ we have that
$$
\norm{Q_{m} - \qms}{\sigma\to \sigma}\leq 2 \frac \sigma m.
$$
Consequently we find that
$$
\norm{Q_{m}}{\sigma\to \sigma} \leq 1 + 2 \frac \sigma m.
$$
\end{lem}

As a consequence of the preceding lemma we obtain

\begin{cor}\label{cor:accretive}
 Suppose that~$\sigma>0$ is chosen such that~$\qms F'(\xnull) \qms$ is
 accretive w.r.t. $\langle.,.\rangle_\sigma$. Then $Q_m F'(\xnull) Q_m +
 \frac{8\sigma\|\xnull\|}m \Id$ is accretive
 w.r.t. $\langle.,.\rangle_\sigma$. Consequently 
 \begin{equation*}
  \|(\alpha\Id + Q_m F'(\xnull) Q_m)\inv\|_\sigma\le \frac2{\alpha}
 \end{equation*}
 and
 \begin{equation*}
  \|(\alpha\Id + Q_m F'(\xnull) Q_m)\inv Q_m F'(\xnull) Q_m\|_\sigma\le 2
 \end{equation*}
 for $\alpha\ge \frac{16\sigma\|\xnull\|}m$.
\end{cor}

Based on these preliminary results, we show that in the setting of Example~\ref{ex:fsigma-constant} the result of Theorem \ref{thm:main} can also
be obtained with the projections~$Q_m:=Q_m^0$ even for $\sigma>0$.

\begin{pro}\label{pro:main}
 Let $\sigma >0$ and $Q_m,\qms$ as in Lemma \ref{lem:qm-qms}.
 We denote the projection space corresponding to $Q_m$ by $X_m^0$. 
 For the noise level~$0< \delta\leq 1$ assume
 that we have~$m \delta\geq 1$. Then the Lavrent\cprime ev-regularized equation
 \begin{equation}\label{eq:reg0}
  \alpha Q_{m}(x_{\ast} - x) + Q_{m}\yd = Q_{m} F(Q_{m}x).
 \end{equation}
has a solution~$\xad\in
X_{m}^0$ with
$$
\norm{\xnull - \xad}{} \leq c \sqrt\delta,
$$
for a suitable parameter choice $\alpha=c\sqrt\delta$ 
with some constant~$c$ (independent of~$\delta$). 
\end{pro}

The proofs of the previous results, Lemma~\ref{lem:qm-qms},
Corollary~\ref{cor:accretive}, and Proposition~\ref{pro:main} are
given in the appendix.

\section{Numerical simulation}
\label{sec:numerics}

Here we are going to highlight the validity of the theoretical
findings. Also, we want to discuss that some of the theoretical
limitations cannot be seen in practical simulations. This concerns the
assumption of the accretivity, i.e., when we require $\sigma$ to be
chosen according to Lemma~\ref{lem:suumary-sigma}.

\subsection{An explicit solver}

We shall show that the discretization described in
Example~\ref{ex:fsigma-constant} leads to a convenient explicit scheme
for solving the discretized equation~(\ref{eq:Qbase}) with~$Q_{m}:=
\qms$. Thus, we fix~$\sigma>0$, and $m\in\bbN$, and we consider the
orthogonal projection~$\qms$ from Example~\ref{ex:fsigma-constant}.


First we simplify the operator $\qms\circ F\circ \qms$. Therefore we define
for  $x\in L_2(0,1)$ , and for~$1\le i\le m$, the quantities
\begin{align*}
f_i(s)& :=
\begin{cases}
e^{\sigma s}, & \text{ for }\frac{i-1}m \le s \le  \frac
  im,\\
0,& \text{ else},
\end{cases}\\
\intertext{and}
x^i& := m \int_{\frac{i-1}m}^{\frac im} e^{-\sigma t}x(t)\de t.
  \end{align*}
In these terms we see that
\[
\qms x = \sum_{i=1}^m x^i f_i,\quad x\in L_{2}(0,1).
\]
We identify $\mathcal R(\qms)$ with a vector in $\bbR^m$ through the bijection
\[ 
x\longrightarrow (x^i)_{1\le i\le m}.
\]
Furthermore, let
\[
h_i(s):=
\begin{cases} 
\frac{i-1}m + s & \text{ for } \frac{i-1}m\le
  s < \frac im\\ \frac{i+1}m - s & \text{ for } \frac{i}m\le s < \frac
  {i+1}m\\0&\text{ else}
\end{cases}
\]
Then for $x\in L_2(0,1)$ we have
\begin{align*}
 (\qms\circ F\circ \qms)(x) =& (\qms\circ F)\left(\sum_{i=1}^m x^i f_i\right)\\
 =& \qms\left(\sum_{i=1}^m \sum_{j=1}^{m-i+1}x^i x^j f_i*f_j\right)\\
 =& \qms\left(\sum_{k=1}^m \sum_{j=1}^{k}x^{k-j+1} x^j(f_{k} + f_{k+1}) h_{k}\right)\\
 =& \frac1{2m} \sum_{k=1}^m \sum_{j=1}^{k} x^{k-j+1} x^j (f_{k} + f_{k+1}).
\end{align*}

Now we can write equation~(\ref{eq}) component-wise as follows.
For $i=1$ we find that $x^{1}$ must satisfy the quadratic equation
$$
\lr{x^{1}}^{2} + 2 m \alpha x^{1} - \lr{2m \lr{\yd}^{1} + 2 m \alpha
  x_{\ast}^{1}} = 0.
$$
Note that $(y^\delta)^i\ge 0$ for all $i=1,\dots,n$,
since we can assume that $y^\delta\ge 0$ (see remark
\ref{rem:data-nonneg}).
The non-negative solution is
\begin{equation}\label{xone}
x^1=-m\alpha +\sqrt{m^2 \alpha^2 + 2m((\yd)^1 + \alpha x_*^1)}.
\end{equation}
For $i=2,\dots,m$ we find that
\begin{equation}\label{eqcomp}
\alpha (x_*^i-x^i) + (\yd)^i - \frac1{2m}\left( \sum_{j=1}^{i}
  x^{i-j+1} x^j + \sum_{j=1}^{i-1} x^{i-j} x^j \right)=0. 
\end{equation}
This can be considered as a linear equation in $x^i$, if all $x_j$ for $j<i$ are already
determined. Rearrangement yields 
\begin{equation}\label{eqre}
 x^i = \frac m{m\alpha + x^1}\left((\yd)^i + \alpha x_*^i - \frac 1 {2m}\left(\sum_{j=1}^{i-1} x^j x^{i-j} + \sum_{j=2}^{i-1} x^j x^{i-j+1}\right)\right).
\end{equation}
Thus we have obtained the piece-wise constant approximation~$\qms x$,
and we now apply \emph{post-smoothing} by a cubic spline. By doing
this appropriately we can retain approximation rate $\frac12$.
Indeed, consider a cubic spline~$S \in S_{4}^{m}$, and let
~$x_\alpha^\delta$ be the above piece-wise constant approximation.
The triangle inequality yields
\begin{equation*}
 \norm{S-x_\alpha^\delta}{} \le \norm{S-x_0}{} + \norm{x_0 - x_\alpha^\delta}{}
\end{equation*}
By the main result, the second summand above  is of order
$\sqrt\delta$. On the other hand, the solution~$\xnull$ can be
approximated by a cubic spline at the rate~$\frac1{m^2}$ (see
\cite[Corollary 6.21]{MR2348176}). 
Hence we can have the left hand side of order $\sqrt\delta$ by
choosing $m\ge\delta^{-\frac14}$.

Thus, for given~$\sigma>0$, $m\in\bbN$, and parameter~$\alpha>0$, this
results in the following simple algorithm to compute a solution 
to~(\ref{eq}), which  requires~$O(n^2)$ operations.
\begin{figure}[ht]
  \centering
  \caption{Description of the algorithm}
\begin{description}
\item[ Input] Data~$\yd$, reference element~$x_{\ast}$.
\item[  Init] Compute~$x^{1}$ from~(\ref{xone}).
\item[  Iter] For $i=2,\dots,m$ solve~(\ref{eqre}) to
  obtain~$x^{2},\dots,x^{m}$.
\item[Smooth] Approximate the function~$\hat f:= \sum_{i=1}^{m} x^{i}f_{i}$ by
  a cubic smoothing spline~$S(\hat f) \in \mathcal S_{4}(\Delta_{m})$.
\item[Output] spline~$S(\hat f) $.
 \end{description}  
  \label{fig:algorithm}
\end{figure}

\subsection{Simulation study}
\label{sec:simulationstudy}

The setup is according to the theoretical study. We fix some
function~$\xnull\colon [0,1]\to \bbR^{+}$, and then we generate data~$\yd= F(\xnull) +
\delta\xi$ on a fine grid (meshsize= $5000$), where $\|\xi\|_\infty = 1$.
For different values of~$\delta=0.04,0.0025$ we
let~$\alpha=\sqrt\delta$.
The theoretical results were based on the accretivity assumption, and
we shall perform simulations, both for $x_0$
that satisfies (1) in Lemma \ref{lem:suumary-sigma}, but also
for $x_0$ violating the accretivity assumption for $\sigma=0$.
The spaces~$\tilde X_{m}$ are chosen both, 
as piece-wise constant functions, and cubic splines, respectively, and
on a grid with  discretization 
level~$m \geq \lceil \frac{1}{\delta}\rceil$ for piece-wise constant functions
and $m \geq \lceil \sqrt[4]{\frac{20}\delta}\rceil$ for cubic splines. 

We then show the reconstructions, for piece-wise constant ansatz
functions, without( left panel) and with (middle panel) post-smoothing
by cubic spline interpolation. We add the 
corresponding reconstructions with cubic splines (right panel).
For the first and the last case we obtain a convergence rate of
order $\frac12$ from the theory (cf.~Theorem~\ref{thm:main}).

\subsection{Accretivity for~$\mathbf {\sigma=0}$}
\label{sec:sigma0}

Here we let
\begin{equation}
  \label{eq:f1}
\xnull(t) = t^{2} -2t + 2,\quad 0\leq t\leq 1.  
\end{equation}

This function is positive, decreasing and convex on the unit interval,
and hence~$\sigma=0$ is appropriate for the
accretivity. Figure~\ref{fig:first-delta04} shows the reconstructions
for the noise level~$\delta=0.04$, whereas
Figure~\ref{fig:first-delta0025} for~$\delta=0.0025$.
\begin{figure}[ht]
  \caption{Simulations for~$\delta=0.04$}
  \centering
   \includegraphics[width= 0.3\textwidth]{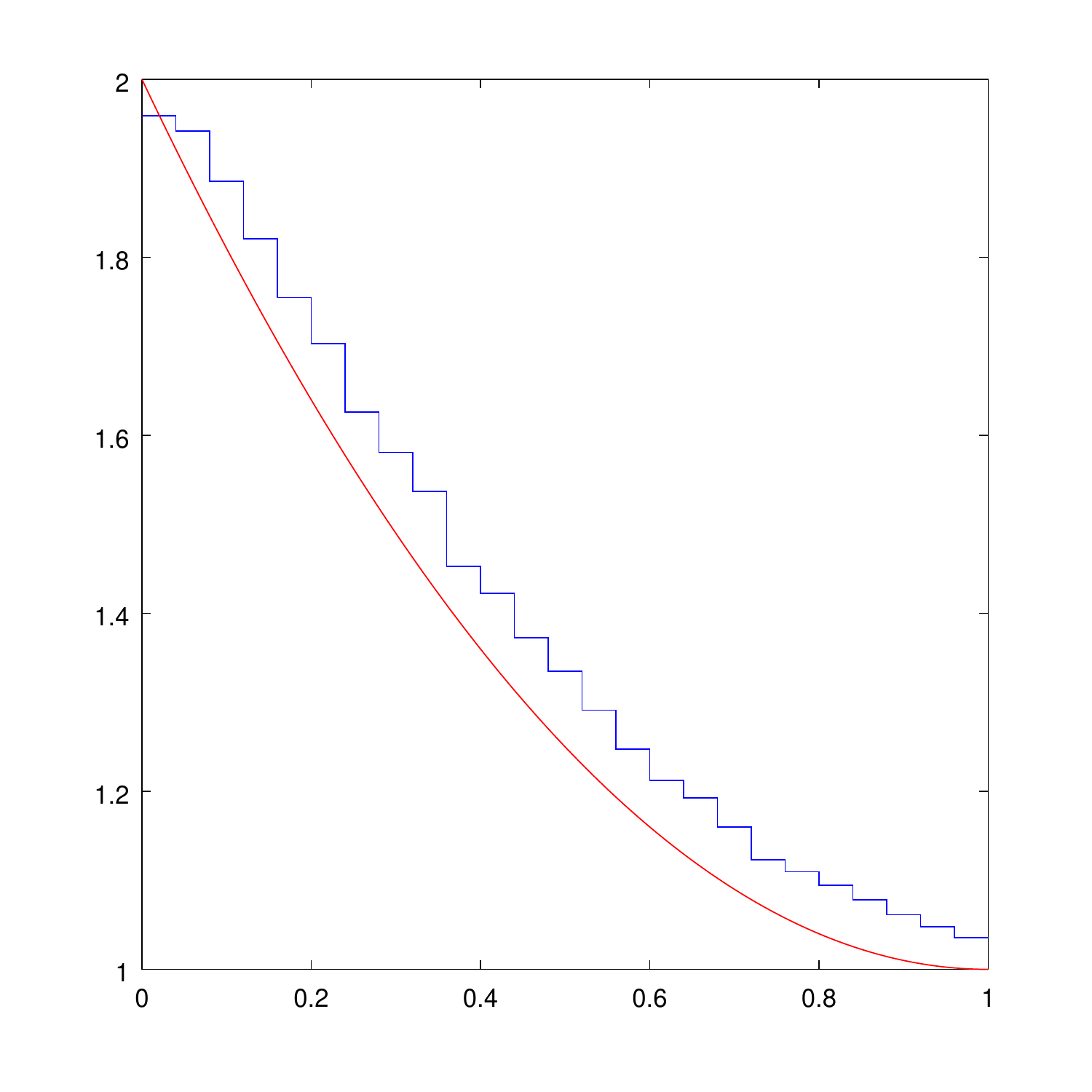}
   \includegraphics[width= 0.3\textwidth]{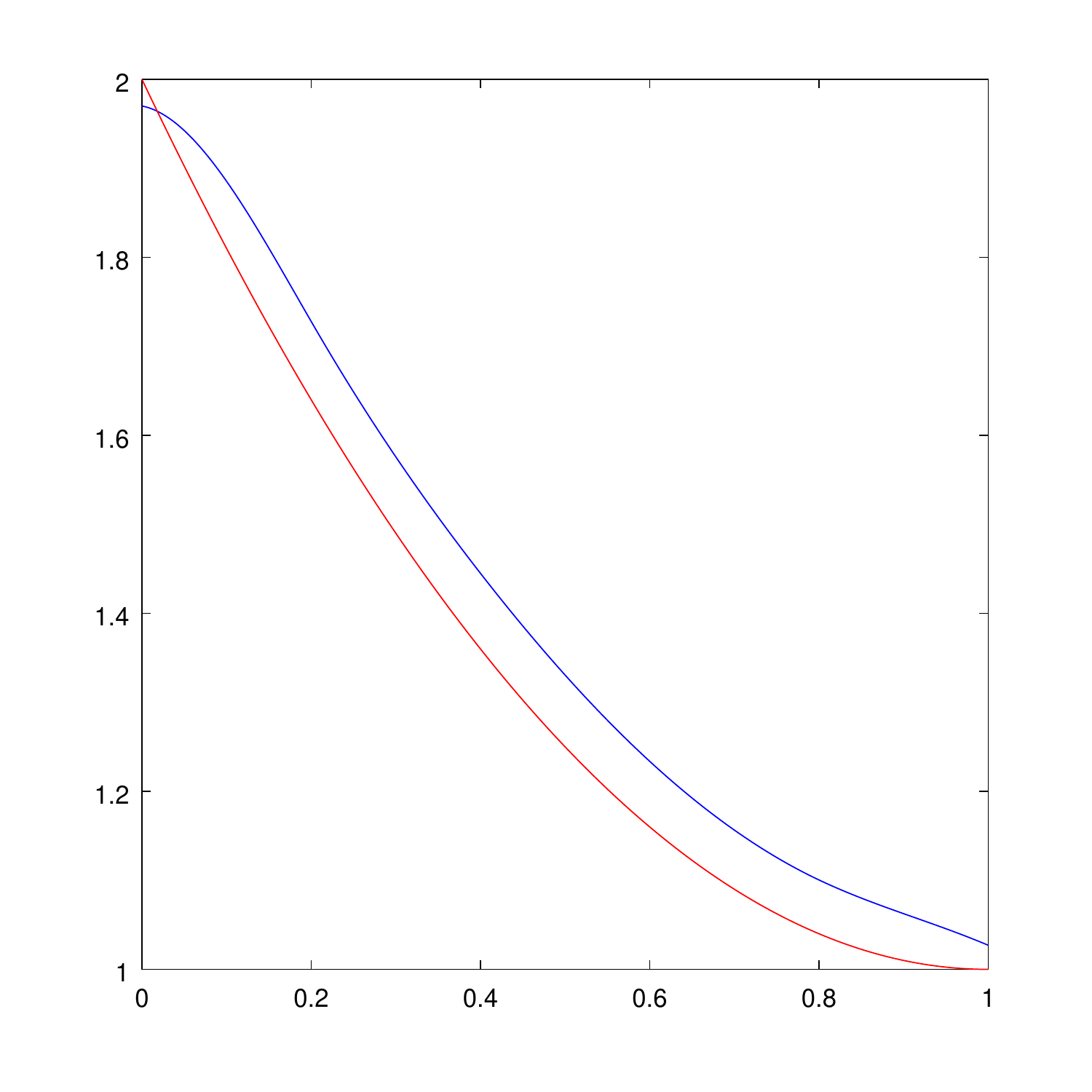}
   \includegraphics[width= 0.3\textwidth]{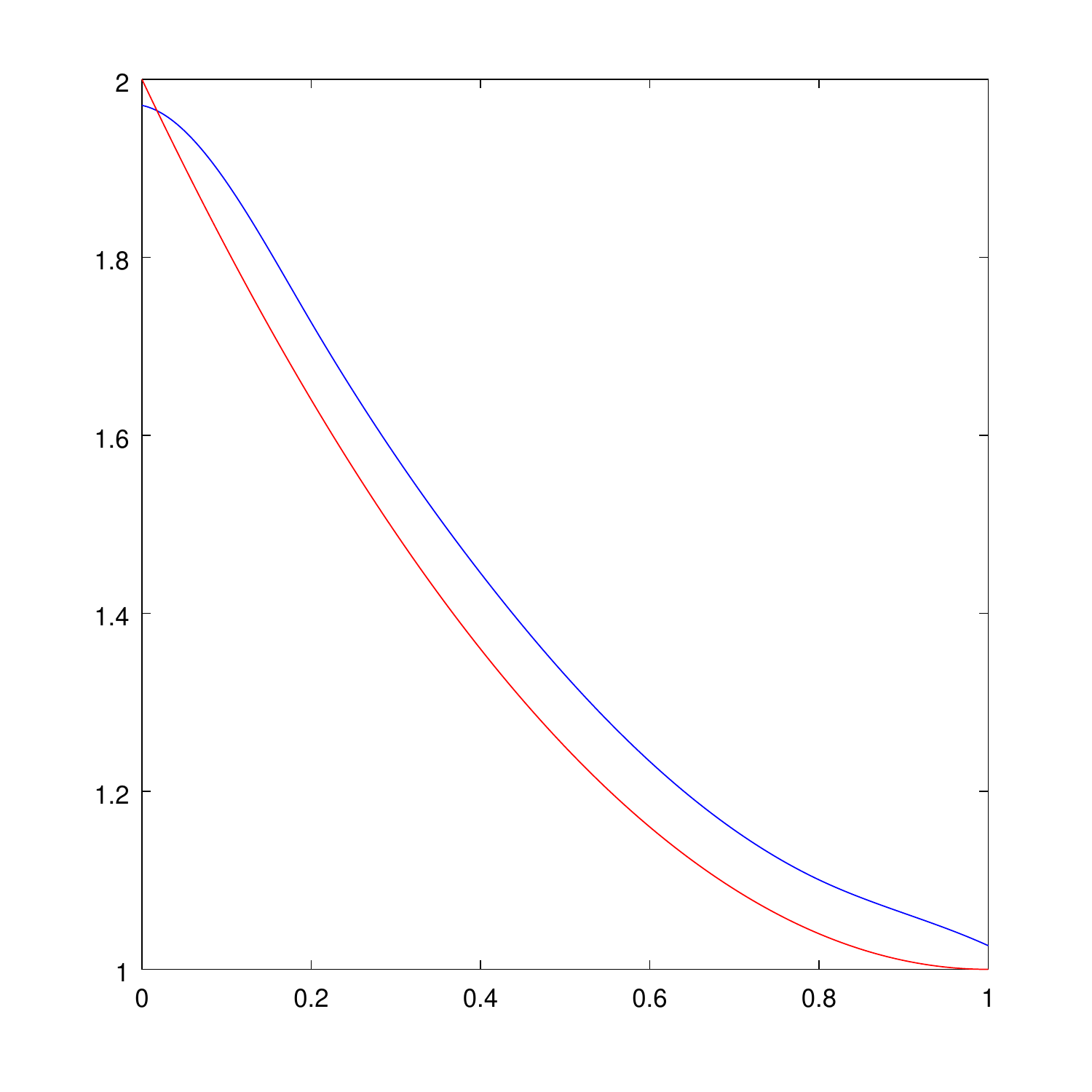}
  \label{fig:first-delta04}
\end{figure}
\begin{figure}[ht]
  \caption{Simulations for~$\delta=0.0025$}
  \centering
   \includegraphics[width= 0.3\textwidth]{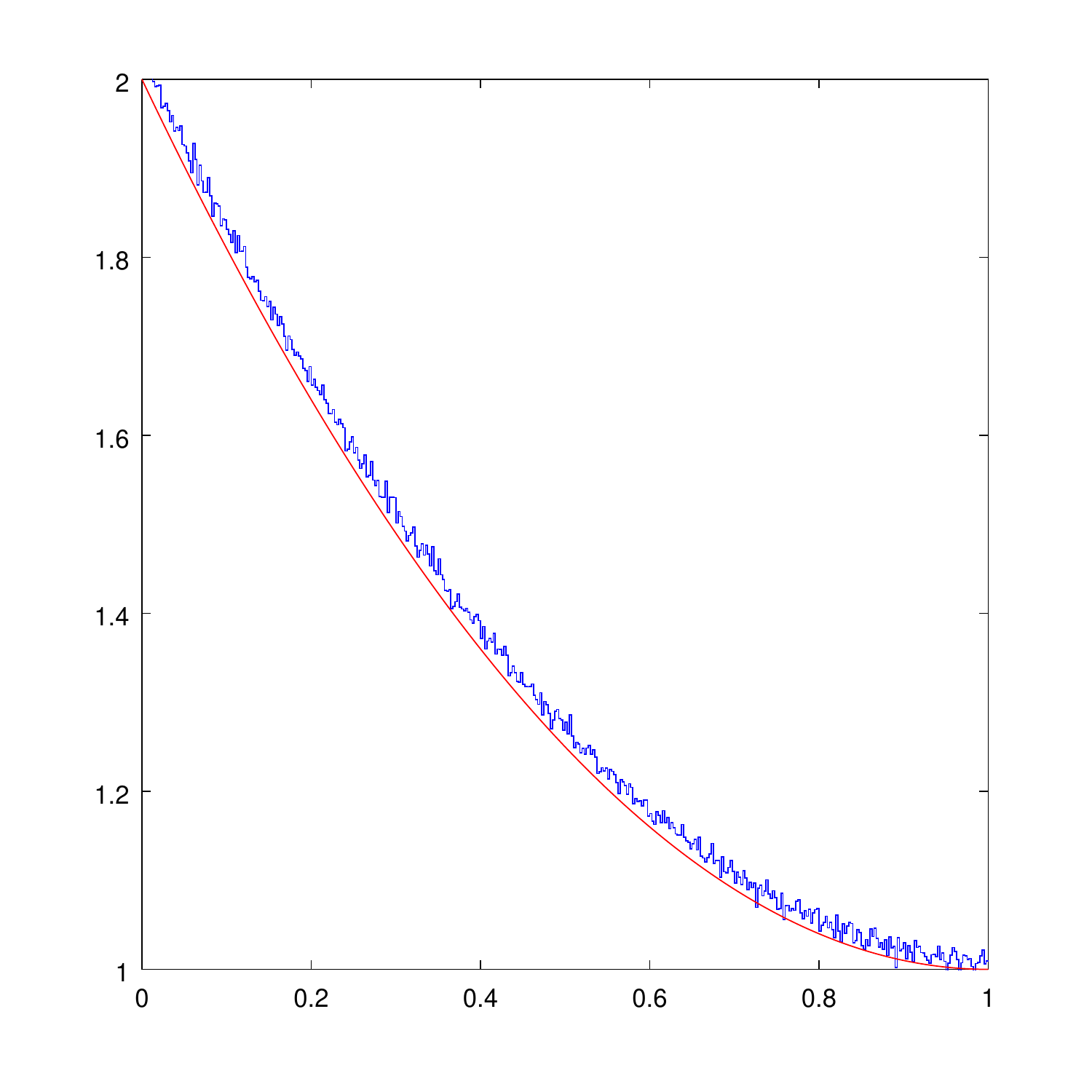}
   \includegraphics[width= 0.3\textwidth]{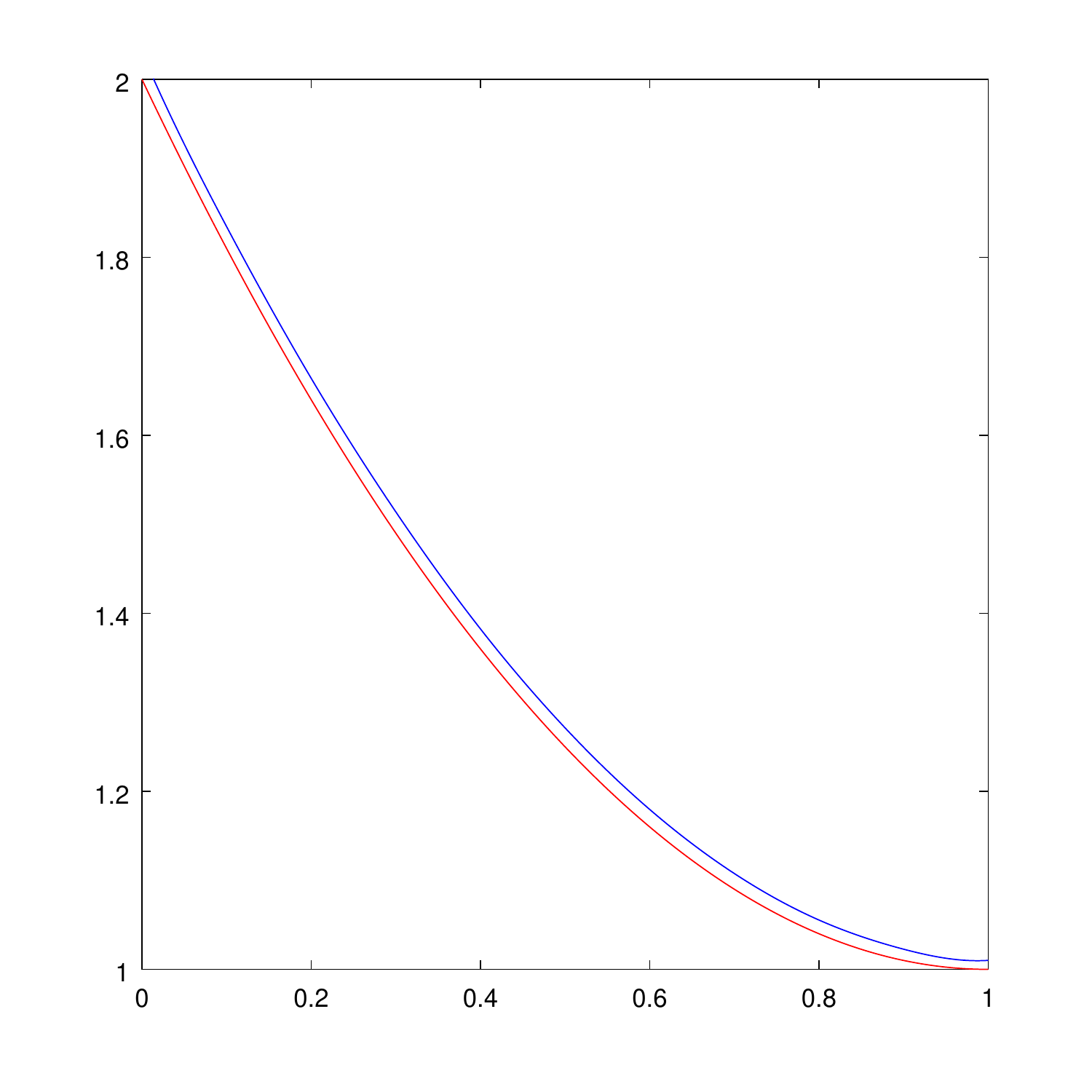}
   \includegraphics[width= 0.3\textwidth]{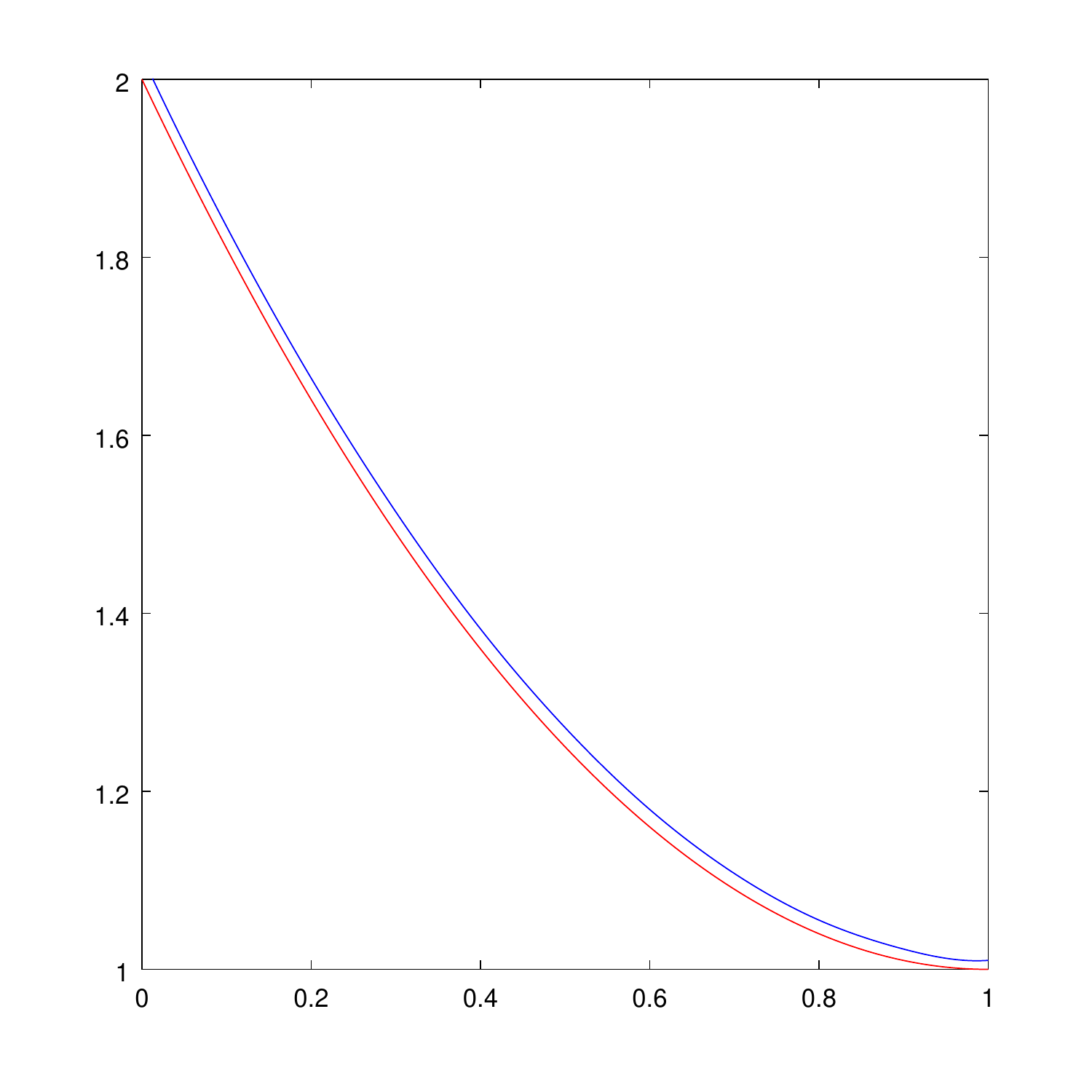}
  \label{fig:first-delta0025}
\end{figure}

\begin{table}[h]
\caption{computation times for function~$\xnull$ from~(\ref{eq:f1})}
\begin{tabular}{r r r r}\toprule
 data error&approximation space&$m$& cpu-time [s]\\\addlinespace\midrule
 $\delta = 0.0025$&piece-wise constant&$400$&$1.68$\\
 &cubic splines&$10$&$1076$\\\midrule
 $\delta = 0.04$&piece-wise constant&$25$&$0.173$\\
 &cubic splines&$5$&$338$\\\bottomrule
\end{tabular}
\label{tab:cpu-f1}
\end{table}

The computation times for the different methods are shown in Table~\ref{tab:cpu-f1}.
Since the time for smoothing the piece-wise constant solution is negligible, the corresponding computation times are not listed here.

Figure \ref{fig:rate1} shows a log-log plot of reconstruction errors of the different reconstruction methods depending on the noise level.

\begin{figure}[ht]
 \caption{$-\ln(\|x_\alpha^\delta-\xnull\|)$ versus $-\ln(\delta)$ for
   piece-wise constant (red),  smoothed piece-wise constant (blue),  and
   cubic splines (black) ansatz} 
 \centering
 \includegraphics[width=5in,height=3in]{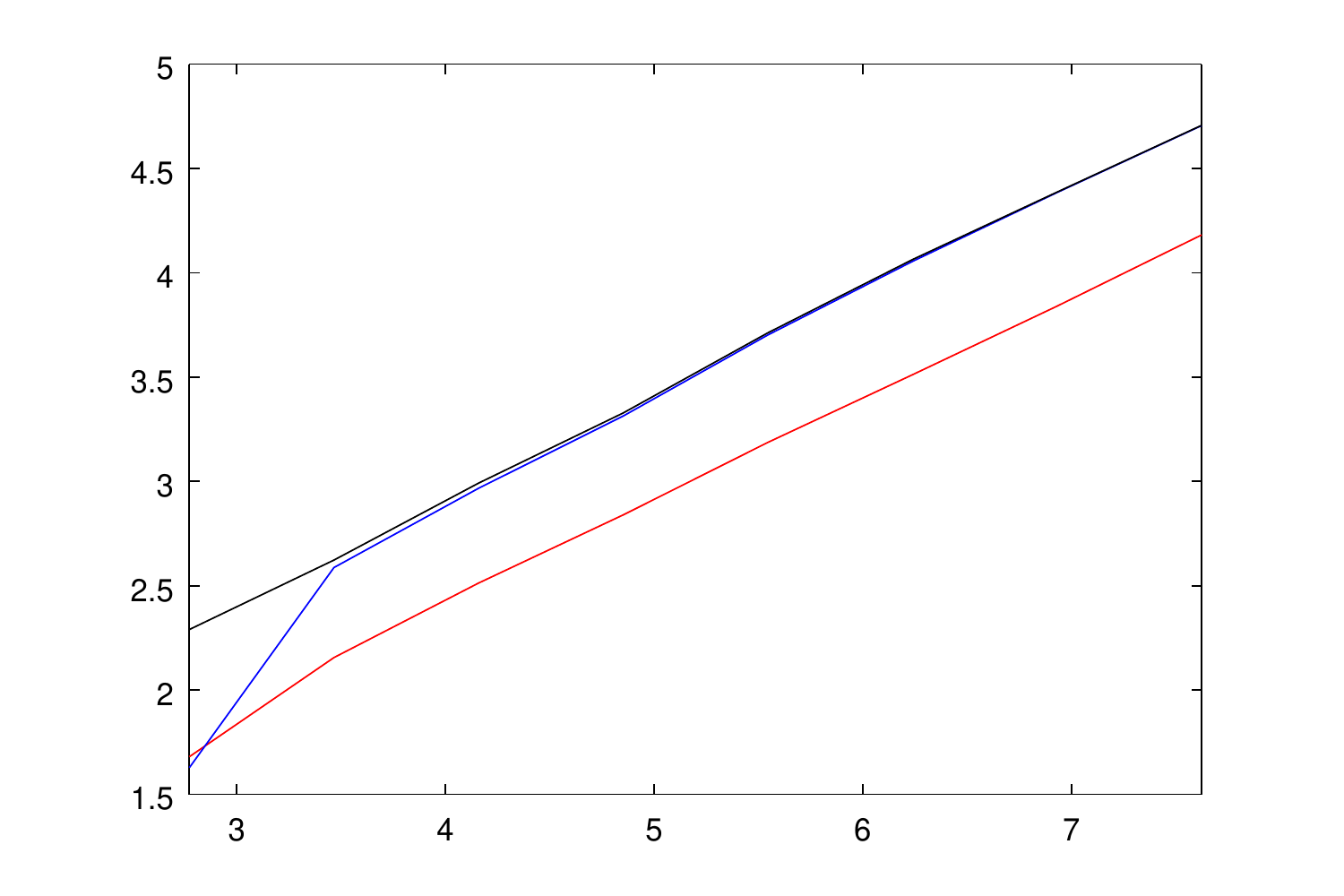}
 \label{fig:rate1}
\end{figure}

We observe that we obtain acceptable reconstructions from all three methods,
where post smoothing of the piece-wise constant reconstruction and cubic spline ansatz yield almost identical results.
Since the computation time for the piece-wise constant ansatz is much
lower, we recommend this method. 
In Figure \ref{fig:rate1} we see that all three methods have a numerical convergence rate of approximately $\frac12$,
where the curves for the last two methods are almost identical for small data error.
\subsection{Accretivity for positive $\sigma$}
\label{sec:sigmapositive}

We consider the function
\begin{equation}
  \label{eq:f2}
\xnull(t) = 2 + \cos(4\pi t),\quad 0\leq t\leq 1.  
\end{equation}
 
The following Figures~\ref{fig:sigma-positive-delta04} \&
\ref{fig:sigma-positive-delta0025}, show the reconstructions.
\begin{figure}[ht]
  \caption{Simulations for~$\delta=0.04$}
  \centering
   \includegraphics[width= 0.3\textwidth]{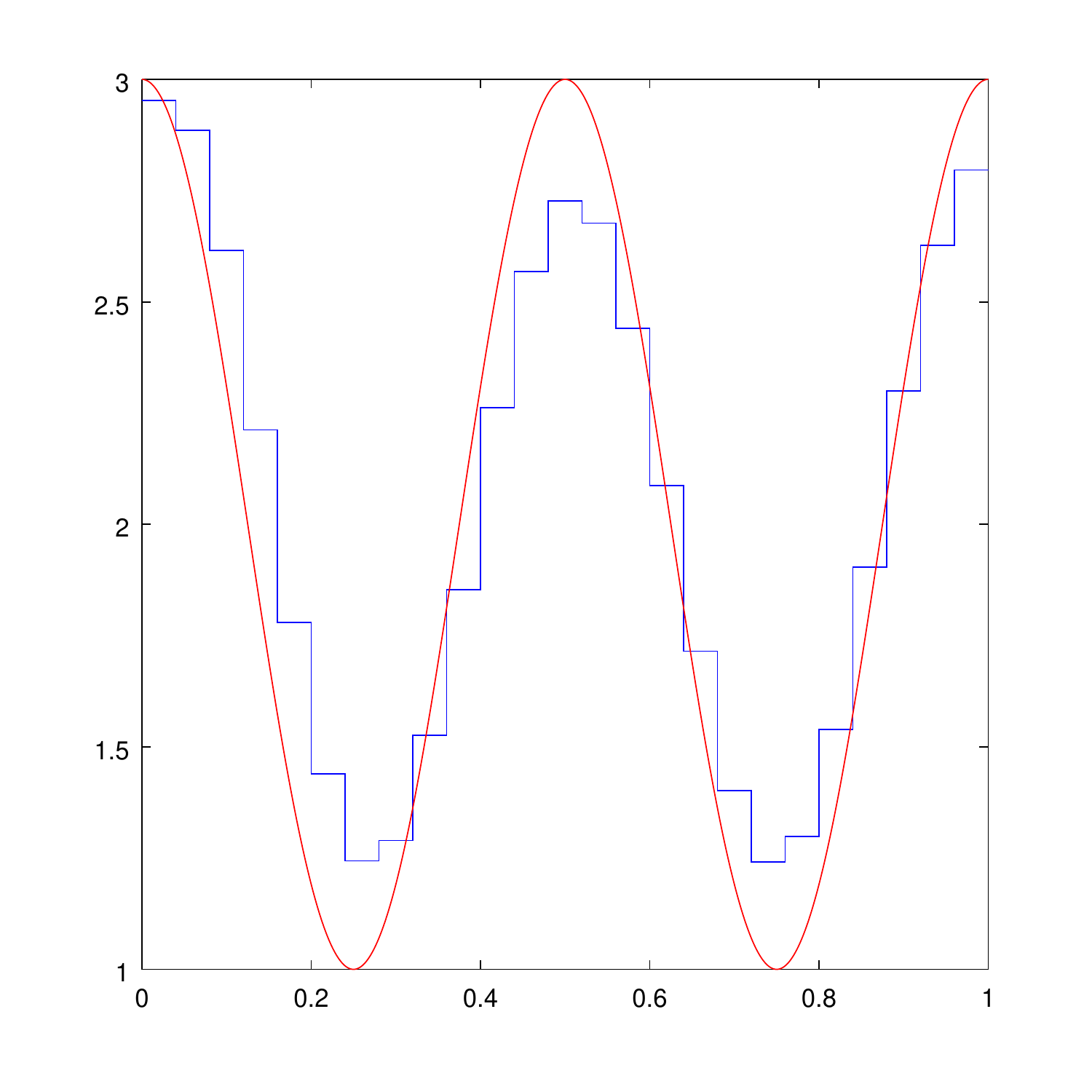}
   \includegraphics[width= 0.3\textwidth]{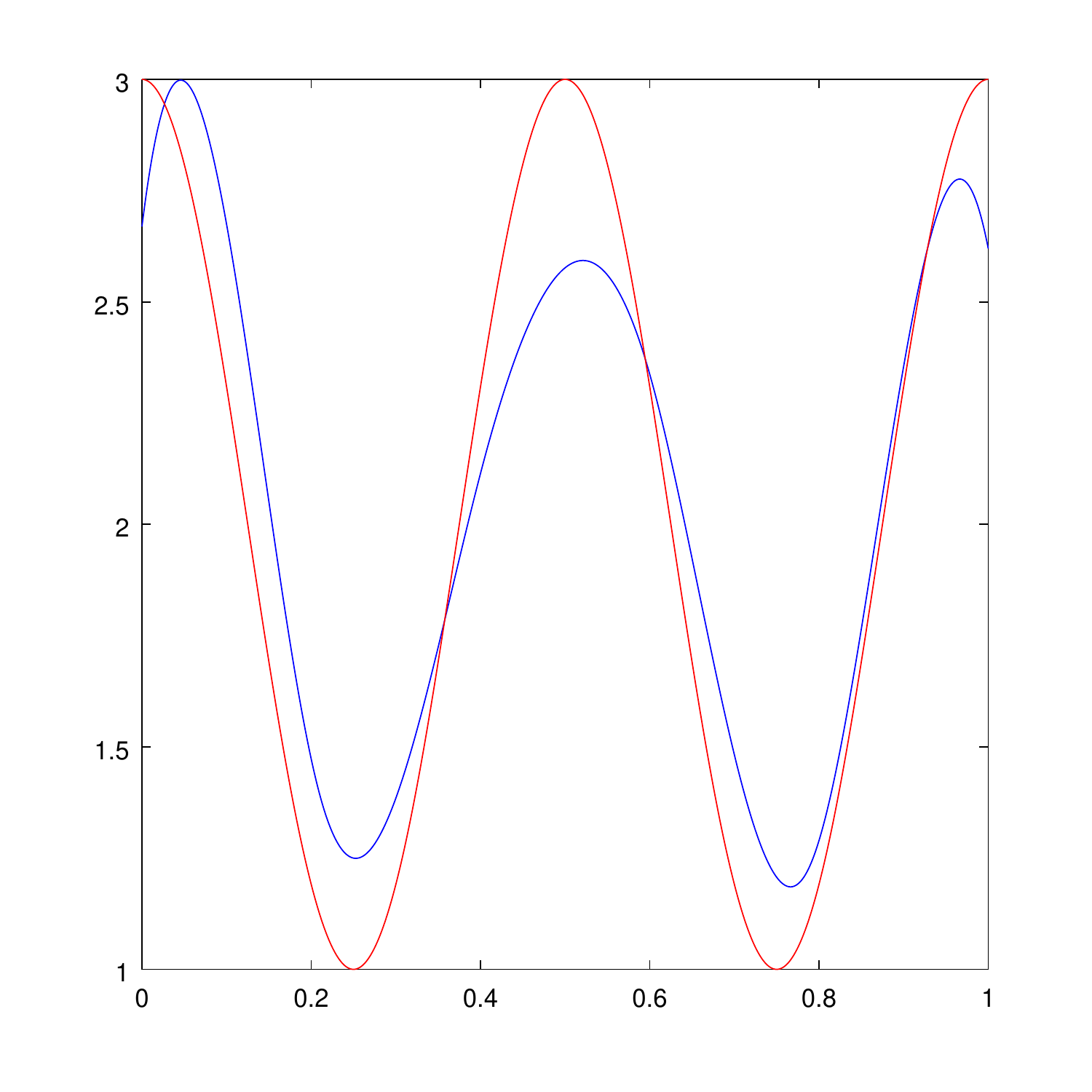}
   \includegraphics[width= 0.3\textwidth]{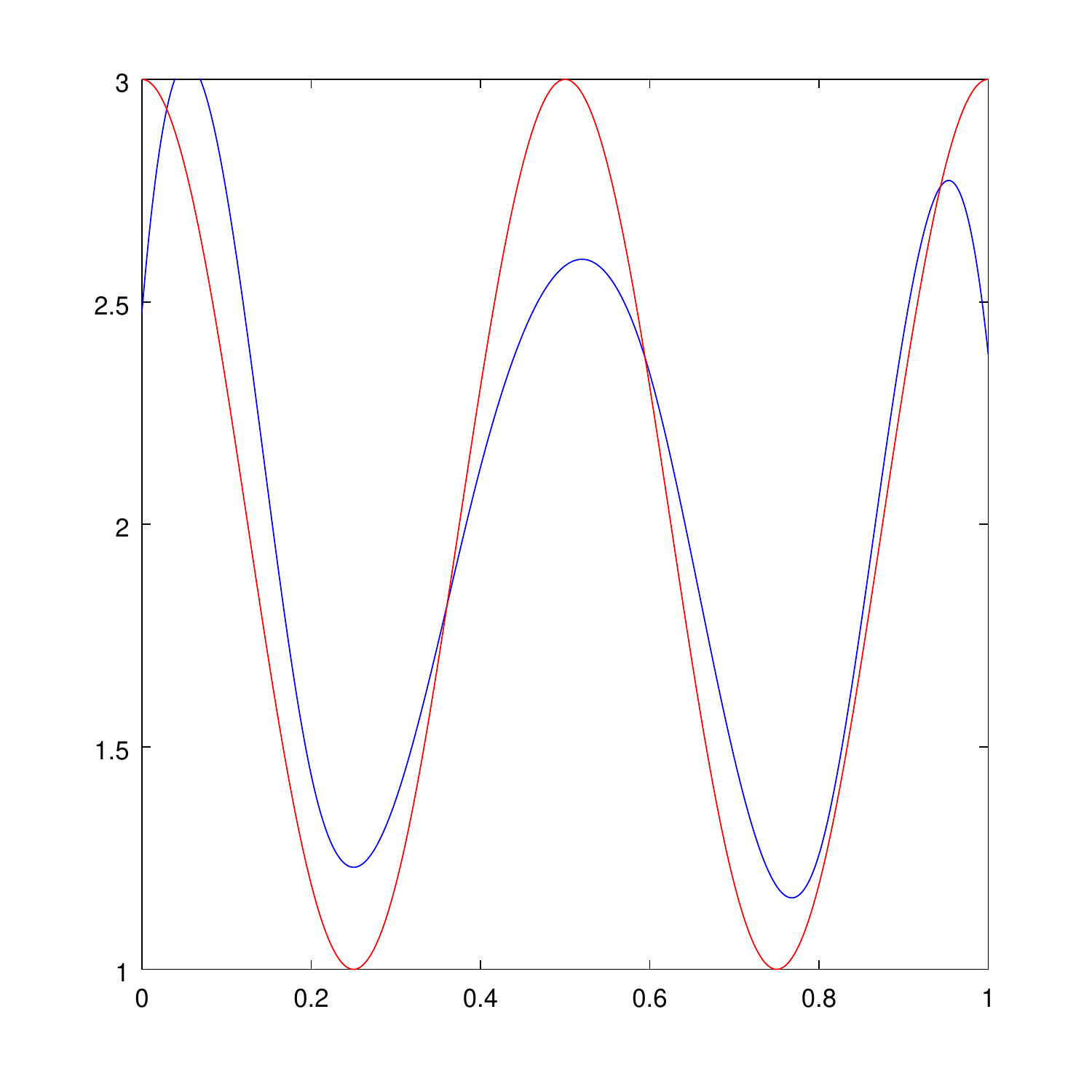}
  \label{fig:sigma-positive-delta04}
\end{figure}
\begin{figure}[h]
  \caption{Simulations for~$\delta=0.0025$}
  \centering
   \includegraphics[width= 0.3\textwidth]{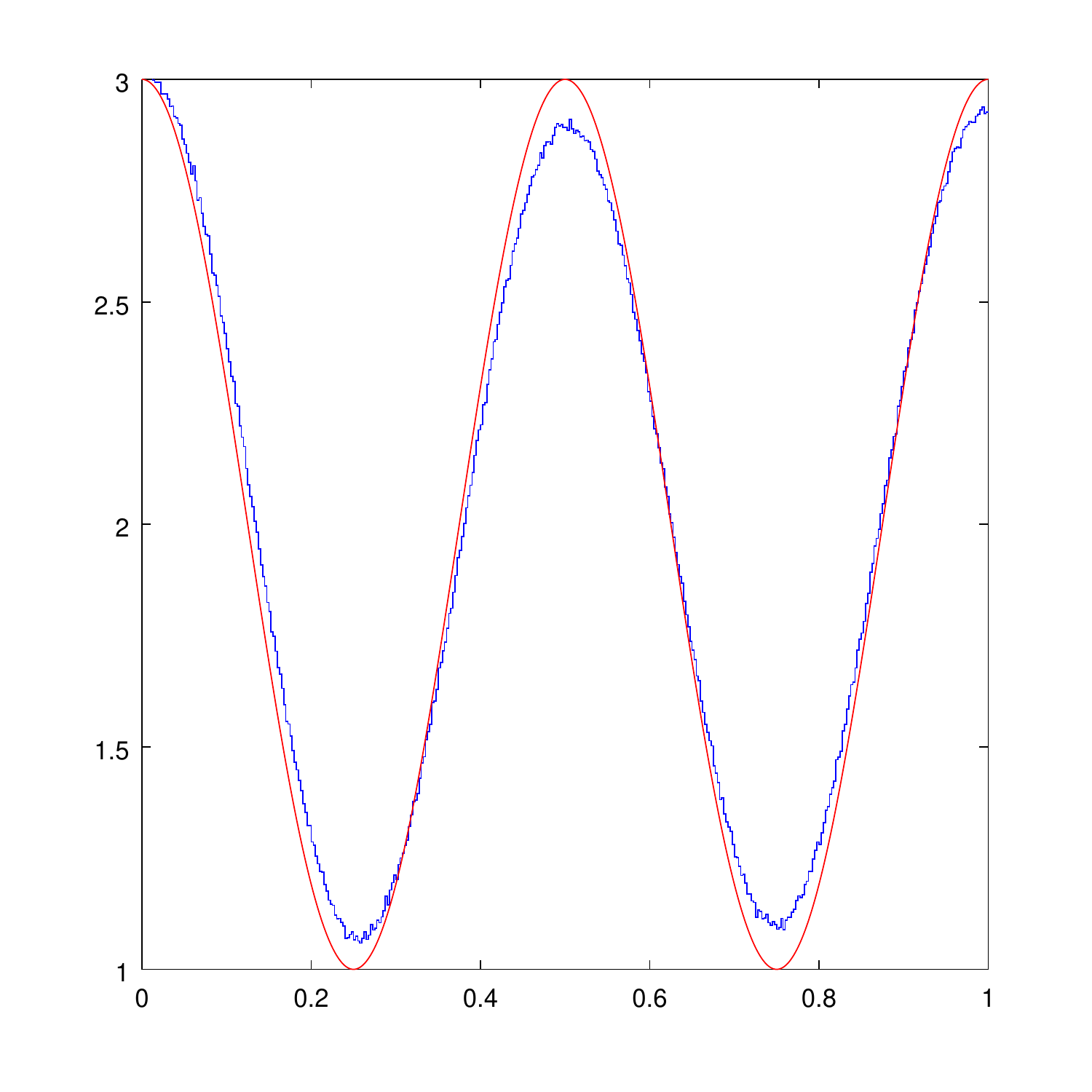}
   \includegraphics[width= 0.3\textwidth]{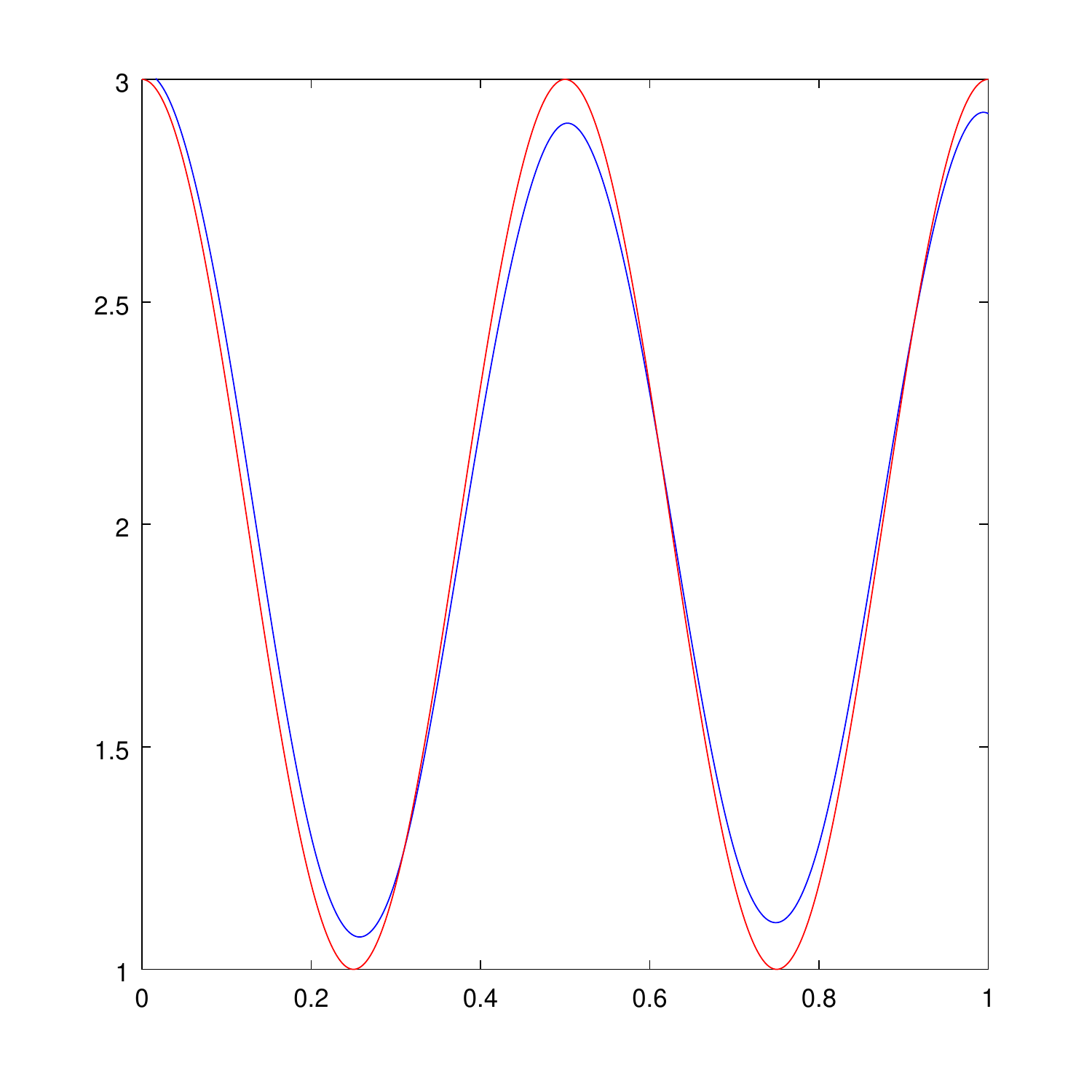}
   \includegraphics[width= 0.3\textwidth]{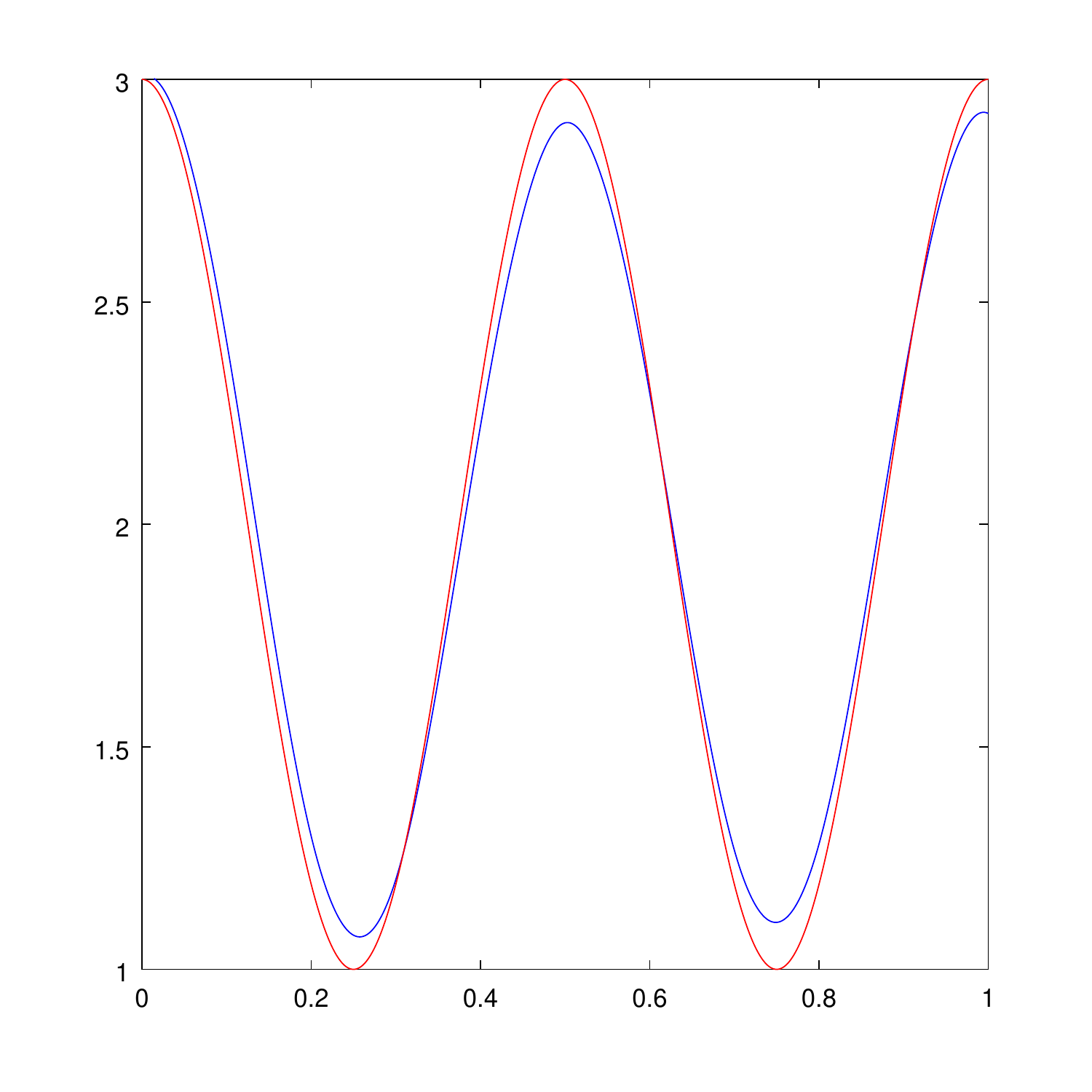}
  \label{fig:sigma-positive-delta0025}
\end{figure}

\begin{table}[h]
\caption{computation times for function~$\xnull$ from~(\ref{eq:f2})}
\begin{tabular}{r r r r}\toprule
 data error&approximation space&$m$& cpu-time [s]\\\addlinespace\midrule
 $\delta = 0.0025$&piece-wise constant&$400$&$1.60$\\
 &cubic splines&$10$&$1003$\\\midrule
 $\delta = 0.04$&piece-wise constant&$25$&$0.176$\\
 &cubic splines&$5$&$319$\\\bottomrule
\end{tabular}
\end{table}

\begin{figure}[h]
 \caption{$-\ln(\|x_\alpha^\delta-\xnull\|)$ versus $-\ln(\delta)$ for
   piece-wise constant (red),  smoothed piece-wise constant (blue),  and
   cubic splines (black) ansatz} 
 \centering
  \includegraphics[width=\textwidth,height=3in]{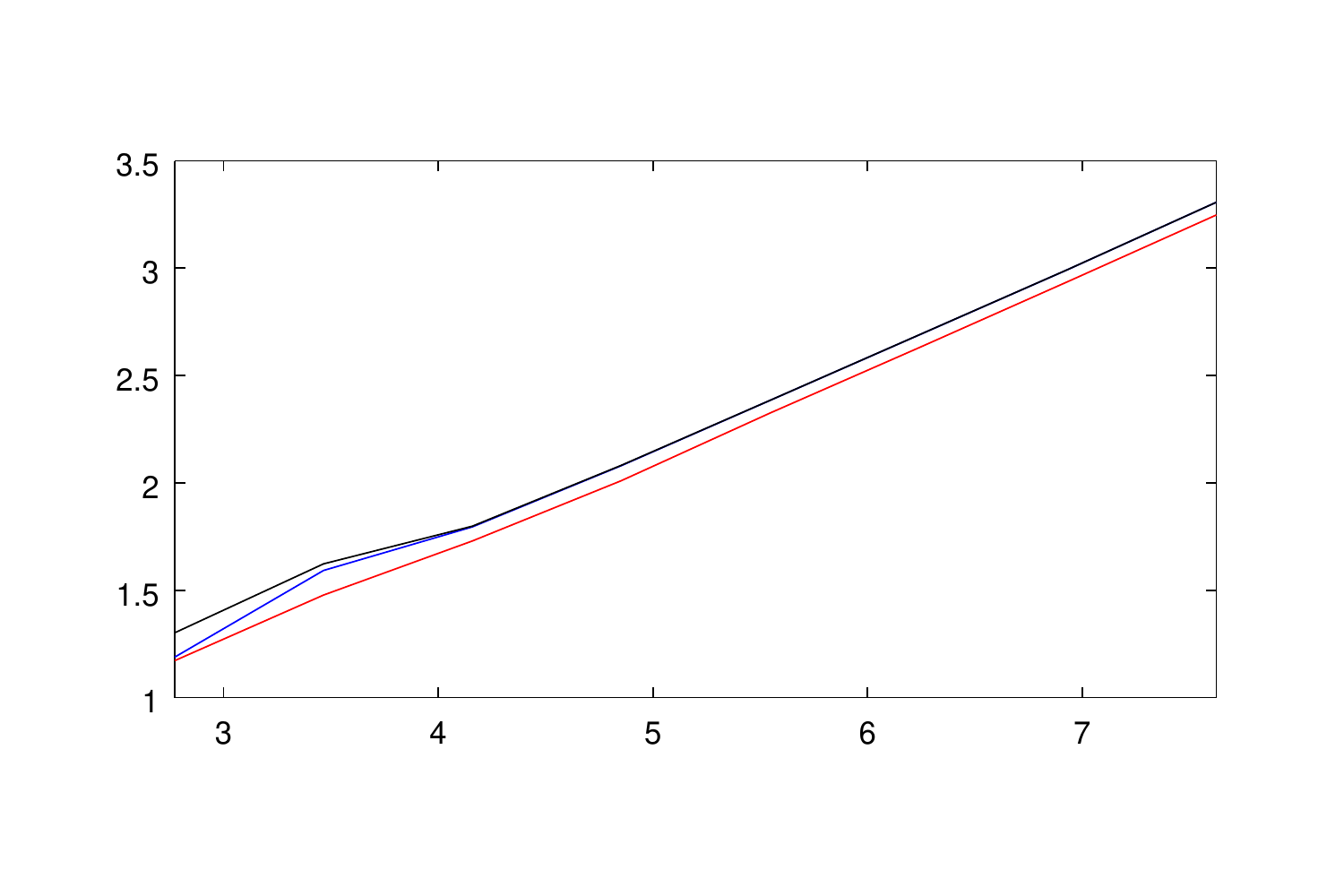}
 \label{fig:rate2}
\end{figure}

In principle we observe the same behavior as for the first test function.
The numerical convergence rates shown in figures \ref{fig:rate2} are
slightly lower than $\frac12$. 
A reason could be that the constant $c$ in the
parameter choice rule $\alpha = c\sqrt\delta$ has been chosen to small.

\clearpage
\appendix
\section{Proofs}

\begin{proof}[Proof of Lemma~\ref{lem:02sigma}]
First, notice that the mappings~$\qms$ are idempotent, and hence
projections onto their ranges.
Further we have the relation between the weighted and unweighted
spaces
 \begin{equation*}
  \|x\|_\sigma = \|D_\sigma\inv x\|,\qquad x\in L_2(0,1).
 \end{equation*}
This yields 
$$
\norm{\qms x}{\sigma}=\norm{D_{\sigma}^{-1}\qms x}{}  
= \norm{Q_{m}   D_{\sigma}^{-1}x}{}
= \norm{D_{\sigma}^{-1}x}{} = \norm{x}{\sigma},\ {x\in D_\sigma(X_m)},
$$
such that the mappings~$\qms$ are partial isometries, and hence
orthogonal projections on their ranges.

 To show the approximation property let $x\in H_1(0,1)$ be arbitrary.
 Then $f_\sigma\inv x \in H_1(0,1)$ because $f_\sigma\inv$ is a smooth
 function. Since $\qm$ satisfies assumption \ref{ass:onQ} we obtain 
 \begin{equation*}
  \|f_\sigma\inv x - \qm(f_\sigma\inv x) \| \le \frac L m
  \|f_\sigma\inv x\|_{H_1}. 
 \end{equation*}
 Moreover we have
 \begin{align*}
   \|f_\sigma\inv x\|_{H_1}^{2} & = \| f_\sigma\inv x \|^{2} + \| (f_\sigma\inv x)' \|^{2}\\
  &= \| f_\sigma\inv x \|^{2} + \| -\sigma f_\sigma\inv x + f_\sigma\inv x'\|^{2}\\
  & \le (1+2\sigma^2)\| f_\sigma\inv x \|^2 + 2\| f_\sigma\inv x'\|^2\\
  & = (1+2\sigma^2) \| x \|_\sigma^2 + 2\| x' \|_\sigma^2\\
  & \le \max\left\{2,1+2\sigma^2\right\} \|x\|_{H_1^\sigma}^{2}. 
 \end{align*}
 Using the above two estimates we derive
 \begin{align*}
  \| x - \qms x\|_\sigma & = \| D_\sigma\inv x - D_\sigma\inv D_\sigma \qm D_\sigma\inv x\|\\
  & = \| f_\sigma\inv x - \qm (f_\sigma\inv x)\|\\
  & \le \frac Lm \|f_\sigma\inv x\|_{H_1}\\
  & \le \max\left\{\sqrt2,\sqrt{1+2\sigma^2}\right\} \frac Lm \|x\|_{H_1^\sigma}.
 \end{align*}
Since we find that~$ \max\left\{\sqrt2,\sqrt{1+2\sigma^2}\right\} \leq
\sqrt 2 \lr{1 + \sigma}$ we can complete the proof.
\end{proof}

\subsection{Proofs for \S~\ref{sec:approx-init}}

Before turning to the proof of Proposition~\ref{pro:x0strich-appr} we
provide representations for $y_{0}$ and its derivatives. Indeed,
since it holds~$y_{0}(0)=0$, we have , from the Taylor expansion that
\begin{align}
y_0(s) & = y_0(0) + y_0'(0)s + \frac 12 y_0''(\xi)s^2 = y_0'(0)s +
         \frac 12 y_0''(\xi)s^2\,\text{ for some $\xi\in [0,s]$}
\label{taylor}\\
 y_0'(s) & = \xnull(0) \xnull(s) + \int_0^s \xnull'(s-t)\xnull(t)\de t = \xnull(0) \xnull(s) +\int_0^s \xnull(s-t) \xnull'(t) \de t\notag,\\
 y_0''(s) & = 2 \xnull(0) \xnull'(s) + \int_0^s \xnull'(s-t) \xnull'(t) \de t\notag.
\end{align}
This shows, among other things that~$y^{\prime\prime}$ is continuous,
but we also conclude that
\begin{equation}
  \label{eq:yprime-bounds}
 \norm{y_{0}^{\prime}}{\infty}\leq 2
\norm{\xnull}{C^{1}(0,1)}^{2}
\quad \text{and}\quad  \norm{y_{0}^{\prime\prime}}{\infty}\leq 3
\norm{\xnull}{C^{1}(0,1)}^{2}.
\end{equation}

\begin{proof}[Proof of Proposition~\ref{pro:x0strich-appr}]
We first derive a uniform bound. We shall use the piece-wise constant
approximations from Example~\ref{ex:fsigma-constant}. For $t\in\Delta$
we then find that
$$
\abs{(Q_{m}y_0 - y_0)(t)} = m \abs{\int_{\Delta}(y_0(\tau) - y_0(t)\de \tau},
$$
which yields the bound
$$
\abs{(Q_{m}y_0 - y_0)(t)} \leq \frac 1 {m} \norm{y_0^{\prime}}{\infty},
$$
by using~$\abs{y_0(\tau) - y_0(t)} = \abs{y_0^{\prime}(\xi) (\tau - t)}\leq
\abs{\Delta} \norm{y_0^{\prime}}{\infty}$. Thus we have 
\begin{equation}
  \label{eq:10}
  \norm{Q_{m} y_0 - y_0 }{\infty}\leq \frac 1 {m}
  \norm{y_0^{\prime}}{\infty} \leq  \frac 2 {m}
  \norm{x_0}{C^{1}(0,1)}^{2},
\end{equation}
by virtue of~(\ref{eq:yprime-bounds}).
Let now $h>0$. We start with~(\ref{taylor}) (s:=h), which yields for
the derivative of $y_{0}$ at zero 
\[
y_0'(0) = \frac{y_0(h) }h - \frac12 y_0''(\xi)h.
\]
Using~(\ref{noise1}), (\ref{eq:yprime-bounds}) and~(\ref{eq:10}) we see that
\begin{align*}
  \left|\frac{[Q_m\yd](h)}h - y_0'(0)\right|& \le \frac{|[Q \yd](h)-[Q
                                              y_0](h)|}h + \frac{|[Q
                                              y_0](h) - y_0(h)|}h\\
& + \frac12 h |y_{0}^{\prime\prime}(\xi)| \le \frac{\delta + \tfrac 2m {\norm{x_0}{C^{1}(0,1)}^{2}}}h + \frac32 \norm{x_0}{C^{1}(0,1)}^{2} h,
\end{align*}
If we set $h={\sqrt \delta}$ this yields
\begin{align*}
\left|\frac1{\sqrt \delta}{[Q_m\yd](\sqrt\delta}) - y_0'(0)\right| 
&\le \lr{1 + \frac {2\norm{x_0}{C^{1}(0,1)}^{2}} {m\delta} + \frac32 \norm{x_0}{C^{1}(0,1)}^{2} }\sqrt\delta\\
&\leq\left(1 + \frac 7 2 \norm{x_0}{C^{1}(0,1)}^{2}\right)\sqrt\delta,
\end{align*}
provided that $m\delta\geq 1$.
Since for $a,b>0$ it holds true that~$\abs{\sqrt a - b}\leq
\tfrac{\abs{a - b^{2}}}{b}$, whereas for $a<0$ we have $b \le \frac {\abs{a-b^2}}b$, we can complete the proof with
letting
\[
a:=
\begin{cases} 
\sqrt{\frac{1}{\sqrt \delta}[Q_m\yd](\sqrt\delta)},& \text{ if }
[Q_m\yd](\sqrt\delta)\ge 0, \\ 
0 ,& \text{else}.
\end{cases}
\]
 and~$b:= \xnull(0)$, such that
$b^{2}= y_0^{\prime}(0)$, cf. the above representations for $y_{0}$ and
its derivatives. 
\end{proof}

\begin{proof}
  [Proof of Proposition~\ref{pro:different-noise}]
Due to (\ref{taylor}) we obtain
\[|y_0(t) - y_0'(0)t| \le \frac 12 y_0(\xi) t^2\qquad\text{for some }\xi\in[0,t].\]
Using (\ref{eq:yprime-bounds}) this gives
\[|y_0(t) - y_0'(0)t| \le \frac 32\|x_0\|^2_{C^1(0,1)} t^2.\]
Integration from $0$ to $h$ yields
\[\left|\int_0^h y_0(t)\de t - \frac12 h^2 y_0'(0)\right| \le \frac12 h^3 \|x_0\|^2_{C^1(0,1)}.\]
Moreover, by the H\"older inequality we have
\begin{align*}
\left|\int_0^h (y_0(t)-\yd(t))\de t \right| &\le \int_0^h |y_0(t)-\yd(t)|\de t\\ 
&\le \sqrt{\int_0^h(y_0(t) - \yd(t))^2 \de t} \cdot \sqrt{\int_0^h\de t} \le \delta \sqrt h.
\end{align*}
The triangle inequality yields
\[\left|\int_0^h \yd(t)\de t - \frac12 h^2 y_0'(0)\right| \le \frac12 h^3 \|x_0\|^2_{C^1(0,1)} + \delta\sqrt{ h}\]
or equivalently
\[\left|\frac2{h^2}\int_0^h \yd(t)\de t - y_0'(0)\right| \le h \|x_0\|^2_{C^1(0,1)} + \delta h^{-\frac32}. \]
Considered as a function of $h$, the right hand side of this inequality
obtains its minimum for $h:=\left(\frac23
  \|x_0\|^2_{C^1(0,1)}\right)^{-\frac25}\delta^{\frac25}$, and then
the assertion is proved.
\end{proof}

\subsection{Proofs for \S~\ref{sec:main-result}}

\begin{proof}[Proof of Proposition~\ref{pro:contraction}]
  We recall the definition of the mapping~$G$
  from~(\ref{eq:G-def}).
We shall first see that there is some $r>0$ such that $\norm{Qx -
  \xnull}{\sigma}\leq r$ yields that~$\norm{G(Qx) -
  \xnull}{\sigma}\leq r$.
By the definition of $G$ we see that
$$
\norm{G(Qx) - \xnull}{\sigma} \leq I_{1} + I_{2} + I_{3} + I_{4}, 
$$
where
\begin{align*}
  I_{1} & := \|H_\alpha\inv Q  F'(\xnull)(\xnull-Q \xnull)\|_\sigma,\\
  I_{2} & := \|H_\alpha\inv Q ( \yd - y )\|_\sigma,\\
  I_{3} & := \|H_\alpha\inv Q F(Q x - \xnull)\|_\sigma,\quad\text{and}\\
  I_{4} &:= \|H_\alpha\inv \alpha (Q x_* -\xnull)\|_\sigma.
\end{align*}
We bound each summand, separately.
By using Lemma~\ref{lem:frechet} and Assumption~\ref{ass:onQ} we find that
\begin{align*}
  I_{1} &\leq
  \norm{H_{\alpha}^{-1}}{\sigma}\norm{F^{\prime}(\xnull)}{\sigma}
  \norm{Q \xnull - \xnull}{\sigma}\leq%
  \frac{2L}{m\alpha}\norm{\xnull}{\sigma}\norm{\xnull}{H_{1}^{\sigma}(0,1)}\\
  &\leq
  \frac{2L}{m\alpha}\norm{\xnull}{H_{1}^{\sigma}(0,1)}^{2}\norm{Q}{\sigma
  \to \sigma}.
\end{align*}
where throughout we use~$\norm{Q}{\sigma  \to \sigma}=1$, since
$Q:L_2^\sigma(0,1)\to L_2^\sigma(0,1)$ is an orthogonal projection by
Assumption~\ref{ass:onQ}. 
By Assumption~\ref{ass:noise}(\ref{noise1}) on the noise we bound 
$$
I_{2} \leq \|H_\alpha\inv\|_\sigma \, \| Q ( \yd - y ) \|_\sigma \le
\frac \delta \alpha.
$$
Next, we see from Lemma~\ref{lem:frechet} that
$$
I_{3} \leq  \|H_\alpha\inv\|_\sigma \, \|F(Q x - \xnull)\|_\sigma \le \frac{\|Q x-\xnull\|_\sigma^2}\alpha.
$$
The bound for~$I_{4}$ is more tedious, and we decompose
\begin{align*}
  I_{4} &\leq \alpha \|H_\alpha\inv(Q ( \xnull(0) -\xnull) + Q(x_* -
          \xnull(0)) +(Q \xnull -\xnull)\|_\sigma\\
 &\leq \alpha \|H_\alpha\inv(Q ( \xnull(0)-\xnull))\|_\sigma + \alpha
   \|H_\alpha\inv Q(x_* -  \xnull(0))\|_\sigma + \alpha \|H_\alpha\inv(Q
   \xnull -\xnull)\|_\sigma.
\end{align*}
Again, we bound separately. The last summand is easily bounded
as
$$
\alpha \|H_\alpha\inv(Q \xnull -\xnull)\|_{\sigma} \leq \frac L m
\norm{\xnull}{H_{1}^{\sigma}} .
$$
For the middle summand we recall that the reference element~$x_{\ast}$
was chosen constant, cf.~(\ref{eq:xast}), such that by
Proposition~\ref{pro:x0strich-appr} we find
$$
 \alpha \|H_\alpha\inv Q(x_* -  \xnull(0))\|_{\sigma}\leq \|Q(x_* -
 \xnull(0))\|_{\sigma}
\leq  \frac{\left(1 +
    {4\norm{\xnull}{C^{1}(0,1)}^{2}}\right)}{\xnull(0)}\sqrt\delta.
$$
It remains to bound the first summand in~$I_{4}$, above. To this end
we use the element~$\omega$ from~(\ref{eq:volterra}) and find that
$$
 \alpha \|H_\alpha\inv(Q ( \xnull(0)-\xnull))\|_\sigma
\leq  \alpha \norm{H_{\alpha}^{-1} Q F^{\prime}(\xnull) \omega}{\sigma}.
$$
We bound the right hand side, again using Lemma~\ref{lem:frechet}, as
\begin{align*}
 \alpha \norm{H_{\alpha}^{-1} Q F^{\prime}(\xnull) \omega}{\sigma} 
& \leq  \alpha \norm{H_{\alpha}^{-1} Q F^{\prime}(\xnull) Q
  \omega}{\sigma} 
+  \alpha \norm{H_{\alpha}^{-1} Q F^{\prime}(\xnull) (\Id - Q)
  \omega}{\sigma}\\
& \leq \alpha \norm{\omega}{\sigma} + 2\norm{Q}{\sigma\to
  \sigma}\norm{\xnull}{\sigma} \norm{(\Id - Q)
  \omega}{\sigma}. 
\end{align*}
By Assumption~\ref{ass:onx} the right hand side in~(\ref{eq:volterra})
is in~$W_{\infty}^{2}(0,1)$, its derivative is in~$H_{1}^{\sigma}$,
and so will be the element~$\omega$, where we refer to~\cite[Proof of
Lem.~4]{MR1766765}. Therefore Assumption~\ref{ass:onQ} gives
$$
 \alpha \norm{H_{\alpha}^{-1} Q F^{\prime}(\xnull) \omega}{\sigma} 
\leq  \alpha\lr{ \norm{\omega}{\sigma} + \frac{2L}{m\alpha}\norm{Q}{\sigma\to
  \sigma}\norm{\xnull}{\sigma} \norm{\omega}{H_{1}^{\sigma}(0,1)}}.
$$
Overall this gives for $I_{4}$ the bound
\begin{align*}
  I_{4} & \leq \alpha \norm{\omega}{\sigma} + \\
&\frac{2L}{m}\norm{\xnull}{\sigma} \norm{\omega}{H_{1}^{\sigma}(0,1)} + \frac L m
\norm{\xnull}{H_{1}^{\sigma}} + \frac{\left(1 +
    {4\norm{\xnull}{C^{1}(0,1)}^{2}}\right)}{\xnull(0)}\sqrt\delta.
\end{align*}
We rearrange terms and write
\begin{equation}
  \label{eq:Gx-x}
\norm{G(Qx) - \xnull}{\sigma} \leq I_{3} + \lr{I_{1} + I_{2} + I_{4}}.  
\end{equation}
The following result proves useful.
\begin{lem}
  \label{lem:quadrat}
Suppose that $a,b,u,v\geq 0$ are such that~$u \leq a v^{2} +b$. If
$4ab<1$ then 
$v\leq r$ implies that~$u\leq r$
for the choice of
 \begin{equation}
    \label{eq:r}
    r= \frac 1 {2a}\lr{1 - \sqrt{1 - 4ab}}.
  \end{equation}
\end{lem}
\begin{proof}
  The assertion holds true if $r$ can be found such that $a r^{2} - r
  + b\leq 0$, and this is the case whenever~$4ab<1$. In this case the
  choice of~$r$ according to~(\ref{eq:r}) does the job.
\end{proof}
We shall apply this fact to the estimate~(\ref{eq:Gx-x}), hence
with~$u:=\norm{G(Qx) - \xnull}{\sigma},  v:= \|Q
x-\xnull\|_{\sigma}, a:= \alpha^{-1}\norm{Q}{\sigma  \to \sigma},
b:= I_{1} + I_{2} + I_{4}$. 
Thus, we aim at arranging the parameter~$\alpha>0$ such that
\begin{align*}
4 \alpha^{-1}\norm{Q}{\sigma  \to \sigma}\lr{I_{1} + I_{2} +
  I_{4}}
& = 4 \norm{Q}{\sigma  \to \sigma}  \norm{\omega}{\sigma}\\
& + 4 \alpha^{-1}\norm{Q}{\sigma  \to \sigma}\lr{I_{1} + I_{2} +
  I_{4} - \alpha  \norm{\omega}{\sigma}}< 1.  
\end{align*}
The first summand is smaller than one under the assumption on~$
\norm{\omega}{\sigma}$,  (note that $\|Q\|_{\sigma\to \sigma}=1$ here), and hence we need to make the second summand,
temporarily denoted by $\res$, 
sufficiently small for an appropriate choice of~$\alpha$. 
Looking at the bounds for~$I_{1},I_{2}$ and $I_{4}$ we find
constants~$C_{1},\dots,C_{5}>0$ such that
$$
\res \leq C_{1}\frac{1}{m\alpha^{2}} + C_{2}
\frac{\delta}{\alpha^{2}} + C_{3} \frac{1}{m\alpha} + C_{4}\frac{\sqrt
\delta}{\alpha} + C_{5}\frac1 {m\alpha}.
$$ 
Now we recall that the bounds were obtained under the assumption
that~$m\delta\geq1$ (cf. Proposition~\ref{pro:x0strich-appr}), and assuming
that the noise level is small, hence assuming that~$\delta\leq 1$,
then we can specify the previous bound to the form
$$
\res \leq \lr{C_{1} + C_{2}}\frac{\delta}{\alpha^{2}} +
\lr{C_{3} + C_{4} + C_{5}}\sqrt{\frac{\delta}{\alpha^{2}}}.
$$
Thus, the remainder~$\res$ can be made arbitrarily small if
$\delta/\alpha^{2}$ is sufficiently small. This can be achieved by
letting~$\alpha := c \sqrt{\delta}$ with a sufficiently large constant~$c>0$.

Under the assumptions made, and in the light of
Lemma~\ref{lem:quadrat} we can find $r>0$ such that
$$
G(B(\xnull,r)\cap \mathcal R Q)\subset B(\xnull,r)\cap \mathcal R Q.
$$
Notice, that the set~$B(\xnull,r)\cap \mathcal R Q$ is compact and
convex, such that the Schauder Fixed Point Theorem,
cf.~\cite[Chapt.~V.9]{MR1070713},  yields the existence
of a fixed point of the continuous mapping~$G$.
\end{proof}
\begin{proof}
  [Proof of Theorem~\ref{thm:main}]
Under the given assumptions, by Proposition~\ref{pro:contraction} the
equation~(\ref{eq}) has a solution, say~$\xad$ in $X_{m}$ which
satisfies~$\xad\in B(\xnull,r)$, where~$r$ is given by
Lemma~\ref{lem:quadrat}, with $a,b$ specified, there.
Obviously, we can bound~$r \leq 1/(2a)$. The specification for~$a$ was given as~$a:=
\norm{Q_{m}}{\sigma\to\sigma}/\alpha$, and it yields that~$r\leq \alpha/(2
\norm{Q_{m}}{\sigma\to\sigma}) \leq \alpha/2$, such that with~$\alpha
= c\sqrt\delta$ (cf. proof of Proposition~\ref{pro:contraction}), we find
$$
r \leq \alpha/2 = \frac c 2 \sqrt\delta.
$$
Notice, that the value of~$c$ depends on the constants~$C_{1},\dots,
C_{5}$, only, and these were dependent on properties of~$\xnull$, but
not on the noise level~$\delta$.
which completes the proof.
\end{proof}

\subsection{Proofs for \S~\ref{sec:extension}}

\begin{proof}[Proof of Lemma~\ref{lem:qm-qms}]
  Let $x\in L_2(0,1)$ and $s\in[0,1)$.
Then there exists $i\in \bbN$, s.t. $s\in[\frac{i-1}m,\frac im)$. Using $\frac\sigma m \le 1$ and the inequality
\begin{equation*}
 |e^x-1| \le 2 |x| \qquad \text{ for } |x| \le 1
\end{equation*}
we estimate
\begin{align*}
 \big|[(Q_m - Q_m^\delta)x](s)\big| & =  m \left| \int_{\frac{i-1}m}^{\frac im} x(t) \de t - e^{\sigma s} \int_{\frac{i-1}m}^{\frac im} e^{-\sigma t} x(t) \de t \right|\\
 & = m e^{\sigma s} \left| \int_{\frac{i-1}m}^{\frac im} (e^{\sigma(t-s)} -1) e^{-\sigma t} x(t) \de t \right|\\
 & \le m e^{\sigma s}  \int_{\frac{i-1}m}^{\frac im} \left|e^{\sigma(t-s)} -1\right| \cdot \left| e^{-\sigma t} x(t) \right| \de t\\
 & \le m e^{\sigma s}  \int_{\frac{i-1}m}^{\frac im} 2\left|\sigma(t-s)\right| \cdot \left| e^{-\sigma t} x(t) \right| \de t\\
 & \le m e^{\sigma s}  \int_{\frac{i-1}m}^{\frac im} 2\frac \sigma m \left| e^{-\sigma t} x(t) \right| \de t\\
 & \le 2 \sigma e^{\sigma s} \left( \int_{\frac{i-1}m}^{\frac im} e^{-2\sigma t} x(t)^2 \de t \right)^{\frac12} \sqrt{\frac 1m},
\end{align*}
where we used the H\"older inequality in the last step.
Now we have
\begin{align*}
 \|(Q_m - Q_m^\delta)x\|_\sigma^2 &= \sum_{i=1}^m \int_{\frac{i-1}m}^{\frac im} e^{-2\sigma s} \big|[(Q_m-Q_m^\sigma)x](s)\big|^2 \de s\\
 &\le  \sum_{i=1}^m \int_{\frac{i-1}m}^{\frac im}  \frac{4\sigma^2}{ m} \left( \int_{\frac{i-1}m}^{\frac im} e^{-2\sigma t} x(t)^2 \de t \right) \de s\\
 &= \frac{4\sigma^2}{ m^2 } \int_0^1 e^{-2\sigma t} x(t)^2 \de t
= \frac{4\sigma^2}{ m^2 } \|x\|_\sigma^2
\end{align*}
and thus
\begin{equation*}
 \|Q_m - Q_m^\sigma\|_{\sigma\to\sigma} \le \frac{2\sigma}m,
\end{equation*}
which proves the first assertion. Finally, in operator
norms we bound
\begin{align*}
 \norm{Q_{m}}{\sigma\to\sigma} \leq \norm{Q_{m} - \qms }{\sigma\to\sigma} + \norm{\qms }{\sigma\to \sigma}  \leq  \frac{2\sigma}m + 1,
\end{align*}
which completes the proof of the lemma.
\end{proof}

\begin{proof}[Proof of Corollary~\ref{cor:accretive}]
 Let $v\in L_2(0,1)$ be arbitrary. Then
 \begin{align*}
  &\spsig{\left(Q_m F'(x) Q_m  + \tfrac{8\sigma\|x\|_\sigma}m \Id\right) v}{v}\\
  &= \spsig{\qms F'(x) \qms v}{v} + \spsig{Q_m F'(x) Q_m v}{v} - \spsig{\qms F'(x) \qms v}{v} + \frac{8\sigma\|x\|_\sigma}m \|v\|_\sigma^2\\
  & \ge \spsig{(Q_m - \qms) F'(x) Q_m v}v + \spsig{\qms F'(x) (Q_m - \qms) v}v  + \frac{8\sigma\|x\|_\sigma}m \|v\|_\sigma^2\\
  & \ge \frac{8\sigma\|x\|_\sigma}m \|v\|_\sigma^2 - \|(Q_m - \qms) F'(x) Q_m v\|_\sigma\cdot\|v\|_\sigma - \|\qms F'(x) (Q_m - \qms) v\|_\sigma\cdot\|v\|_\sigma\\
  & \ge \frac{8\sigma\|x\|_\sigma}m \|v\|_\sigma^2 - 2 \|Q_m - \qms\|_{\sigma\to\sigma}\cdot \|F'(x)\|_{\sigma\to\sigma} \cdot \|v\|_\sigma^2.
 \end{align*}
Now we use Lemma~\ref{lem:frechet} and Lemma~\ref{lem:qm-qms} to see that
$$
2 \|Q_m - \qms\|_{\sigma\to\sigma}\cdot \|F'(x)\|_{\sigma\to\sigma} \cdot \|v\|_\sigma^2
\leq \frac{8\sigma}m \norm{x}{\sigma}\norm{v}{\sigma}^{2},
$$
 which implies the accretivity of $Q_m F'(x) Q_m + \frac{8\sigma\|x\|}m \Id$. 
 For the sake of readability we set $J:= Q_m F'(x) Q_m$. Let now $\alpha \ge \frac{16\sigma\|x\|}m$.
 Then $\frac\alpha2 \ge \frac{8\sigma\|x\|}m$, thus $G := J +
 \frac\alpha 2 \Id$ is accretive. We
 recall~\cite[Eq.~(14)]{MR1766765}, 
 which asserts that for $\beta>0$ we have
 \begin{equation*}
  \|(\beta\Id + G)\inv\|_\sigma \le \frac1\beta \quad \text{and}\quad
  \|(\beta \Id + G)\inv G\|_\sigma \le 1. 
 \end{equation*}
 We use this to conclude that
 \begin{equation*}
  \|(\alpha\Id + J)\inv\|_\sigma = \| \big(\tfrac\alpha 2 \Id + (\tfrac\alpha2\Id + J)\big)\inv\|_\sigma  \le \frac2{\alpha},
 \end{equation*}
 and
 \begin{align*}
\|
&(\alpha\Id+ J)\inv J\|_\sigma = \|\big(\tfrac\alpha2\Id + (\tfrac\alpha2\Id + J)\big)\inv (\tfrac\alpha 2 \Id + J - \tfrac\alpha2 \Id)\|_\sigma\\
& \le \|\big(\tfrac\alpha2\Id + (\tfrac\alpha2\Id + J)\big)\inv (\tfrac\alpha 2 \Id + J) \|_\sigma + \|\big(\tfrac\alpha2\Id + (\tfrac\alpha2\Id + J)\big)\inv \tfrac\alpha2 \Id\|_\sigma\\
  & \le 1 + \frac2\alpha \cdot \frac\alpha2 = 2,
 \end{align*}
and the proof is complete.
\end{proof}

\begin{proof}[Sketch of the proof of Proposition \ref{pro:main}]
 By Lemma \ref{lem:suumary-sigma}, it exists a $\sigma_1>0$ s.t.
 \begin{equation*}
 \|\omega\|_{\sigma_1} < \frac 1{48}.
 \end{equation*}
Since  by Lemma \ref{lem:qm-qms} we have with $m \ge {\sigma_1}$ that
 \begin{equation*}
  \|Q_m\|_{\sigma_1\to \sigma_1}\le 1 + 2 \frac {\sigma_1} m \le 3
 \end{equation*}
and hence we find that
 \begin{equation*}
  \|\omega\|_{\sigma_1}\cdot \|Q_m\|_{\sigma_1\to \sigma_1}< \frac 1{16}.
 \end{equation*}
 In the following we will abbreviate~$Q_{m}$ by~$Q$.
 Basically we follow the proofs of Proposition \ref{pro:contraction}
 and Theorem \ref{thm:main}, but the norm~$\norm{Q}{\sigma\to\sigma}$
 is no longer equal to one.
 Analogous to (\ref{eq:G-def}) we write equation (\ref{eq:reg0}) in fix-point form as
 \begin{equation*}
  G(Qx) = Qx
 \end{equation*}
 with
 \begin{equation*}
  G(Q x) := \xnull + H_\alpha\inv \Big( Q \big( F'(\xnull)(\xnull-Q \xnull) + \yd - y- F(Q_mx-\xnull )\big) + \alpha (Q  x_*-\xnull) \Big),
 \end{equation*}
 where
 \begin{equation*}
  H_\alpha:= Q F'(x_0)Q.
 \end{equation*}
 As in the proof of Prop. \ref{pro:contraction} we have
 $$
\norm{G(Qx) - \xnull}{\sigma_1} \leq I_{1} + I_{2} + I_{3} + I_{4}, 
$$
where
\begin{align*}
  I_{1} & := \|H_\alpha\inv Q  F'(\xnull)(\xnull-Q \xnull)\|_{\sigma_1},\\
  I_{2} & := \|H_\alpha\inv Q ( \yd - y )\|_{\sigma_1},\\
  I_{3} & := \|H_\alpha\inv Q F(Q x - \xnull)\|_{\sigma_1}\quad\text{and}\\
  I_{4} &:= \|H_\alpha\inv \alpha (Q x_* -\xnull)\|_{\sigma_1}.
\end{align*}
The following estimations are almost the same as before, except that we
assume
\begin{equation}\label{eq:ass:alpha}
\alpha\ge \frac{16\sigma\|x\|}m
\end{equation}
and make use of Corollary \ref{cor:accretive}.
This yields
\begin{align*}
  I_{1} &\leq
  \frac{4L}{m\alpha}\norm{\xnull}{H_{1}^{\sigma_1}(0,1)}^{2}\norm{Q}{\sigma_1
  \to \sigma_1},\\
I_{2} &\leq  2\frac \delta \alpha\norm{Q}{\sigma_1
  \to \sigma_1},\\
I_{3} &\leq  2\frac{\|Q x-\xnull\|_{\sigma_1}^2}\alpha\norm{Q}{\sigma_1
  \to \sigma_1}
\end{align*}
and
\begin{align*}
  I_{4} & \leq 2\alpha \norm{\omega}{\sigma_1} + \frac{4L}{m}\norm{Q}{\sigma_1\to
  \sigma_1}\norm{\xnull}{\sigma_1} \norm{\omega}{H_{1}^{\sigma_1}(0,1)}\\
  &+ \frac {2L} m \norm{\xnull}{H_{1}^{\sigma_1}} + \frac{2\left(1 +
    {4\norm{\xnull}{C^{1}(0,1)}^{2}}\right)}{\xnull(0)}\sqrt\delta\norm{Q}{\sigma_1
  \to \sigma_1}.
\end{align*}
Now we want to apply Lemma \ref{lem:quadrat}
with the parameters $u:=\norm{G(Qx) - \xnull}{\sigma_1},  v:= \|Q
x-\xnull\|_{\sigma_1}, a:= 2\alpha^{-1}\norm{Q}{\sigma_1  \to \sigma_1},
b:= I_{1} + I_{2} + I_{4}$. 
Again we have to ensure that $4ab<1$. This is equivalent to
\begin{align*}
 16 \norm{Q}{\sigma_1  \to \sigma_1}  \norm{\omega}{\sigma_1} + 8 \alpha^{-1}\norm{Q}{\sigma_1 \to \sigma_1}\lr{I_{1} + I_{2} +
  I_{4} - 2\alpha  \norm{\omega}{\sigma_1}}< 1.  
\end{align*}
Due to the choice of $\sigma_1$ we have
\begin{equation*}
 16 \norm{Q}{\sigma_1  \to \sigma_1}  \norm{\omega}{\sigma_1}<1
\end{equation*}
and the second summand, denoted by $\res$ can be estimated as
$$
\res \leq C_{1}\frac{1}{m\alpha^{2}} + C_{2}
\frac{\delta}{\alpha^{2}} + C_{3} \frac{1}{m\alpha} + C_{4}\frac{\sqrt
\delta}{\alpha} + C_{5}\frac1 {m\alpha}.
$$ 
with constants $C_1,\dots,C_5$.
With the assumptions $m\delta \ge 1$ and $\delta \le 1$ we derive
$$
\res \leq \lr{C_{1} + C_{2}}\frac{\delta}{\alpha^{2}} +
\lr{C_{3} + C_{4} + C_{5}}\sqrt{\frac{\delta}{\alpha^{2}}}.
$$
Hence $\res$ can be made arbitrarily small by setting $\alpha=c\sqrt\delta$ with $c$ big enough and
$c\ge \frac{16\sigma\|x\|}{m\sqrt\delta}$, which ensures (\ref{eq:ass:alpha}). If $c$ is now chosen in such a way
that 
\begin{equation*}
 16 \norm{Q}{\sigma_1  \to \sigma_1}  \norm{\omega}{\sigma_1} + \res < 1
\end{equation*}
we apply Lemma \ref{lem:quadrat} and obtain
$$
G(B(\xnull,r)\cap \mathcal R Q)\subset B(\xnull,r)\cap \mathcal R Q.
$$
for $r= \frac 1 {2a}\lr{1 - \sqrt{1 - 4ab}}$. Since $B(\xnull,r)\cap \mathcal R Q$ is a
compact and convex set, we obtain that $G$ has a fixed point by the Schauder Fixed Point Theorem.
The remainder of the proof is analogous to the proof of Theorem \ref{thm:main}.
\end{proof}

\def\cprime{$'$}

\end{document}